\documentclass[10pt]{article}

\usepackage{enumerate,xspace}
\usepackage{amsmath,amssymb,wasysym}
\usepackage[all]{xy}
\usepackage{proof}
\usepackage[svgnames]{xcolor}
\usepackage{pict2e} 
\usepackage{tikz}
\usepackage{stmaryrd} 
\usepackage{mathtools}
\usepackage{latexsym}
\usepackage{dsfont}
\usepackage{multicol}
\usepackage{cmll}
\usepackage{multirow}
\usepackage{longtable}
\usepackage{todonotes}
\usepackage{hyperref}
\usepackage{adjustbox} 

\newcommand{\fitline}[1]{%
  \adjustbox{max width=0.93\linewidth}{\ensuremath{\displaystyle #1}}%
}

\makeatletter
\newif\ifbmatrix@small
\renewenvironment{bmatrix}
  {%
    \if@display
      \bmatrix@smallfalse
      \left[\env@matrix
    \else
      \bmatrix@smalltrue
      \left[\begin{smallmatrix}
    \fi
  }
  {%
    \ifbmatrix@small
      \end{smallmatrix}\right]%
    \else
      \endmatrix\right]%
    \fi
  }
\makeatother

\usepackage{hyperref} 
\hypersetup{
    colorlinks,
    citecolor=red,
    filecolor=red,
    linkcolor=blue,
    urlcolor=red
}

\newtheorem{observation}{Remark}[section]
\newtheorem{lemma}[observation]{Lemma}  
\newtheorem{theorem}[observation]{Theorem}
\newtheorem{defi}[observation]{Definition}
\newtheorem{example}[observation]{Example}

\newtheorem{proposition}[observation]{Proposition}

\newcommand{\wand}{\ensuremath{
  \mathrel{\vbox{\offinterlineskip\ialign{
    \hfil##\hfil\cr
    $\star$\cr
    \noalign{\kern-1ex}
    $\vert$\cr
}}}}}

\makeatletter

\newdimen\w@dth

\def\setw@dth#1#2{\setbox\z@\hbox{\scriptsize $#1$}\w@dth=\wd\z@
\setbox\@ne\hbox{\scriptsize $#2$}\ifnum\w@dth<\wd\@ne \w@dth=\wd\@ne \fi
\advance\w@dth by 1.2em}

\def\t@^#1_#2{\allowbreak\def\n@one{#1}\def\n@two{#2}\mathrel
{\setw@dth{#1}{#2}
\mathop{\hbox to \w@dth{\rightarrowfill}}\limits
\ifx\n@one\empty\else ^{\box\z@}\fi
\ifx\n@two\empty\else _{\box\@ne}\fi}}
\def\t@@^#1{\@ifnextchar_ {\t@^{#1}}{\t@^{#1}_{}}}

\def\t@left^#1_#2{\def\n@one{#1}\def\n@two{#2}\mathrel{\setw@dth{#1}{#2}
\mathop{\hbox to \w@dth{\leftarrowfill}}\limits
\ifx\n@one\empty\else ^{\box\z@}\fi
\ifx\n@two\empty\else _{\box\@ne}\fi}}
\def\t@@left^#1{\@ifnextchar_ {\t@left^{#1}}{\t@left^{#1}_{}}}

\def\two@^#1_#2{\def\n@one{#1}\def\n@two{#2}\mathrel{\setw@dth{#1}{#2}
\mathop{\vcenter{\hbox to \w@dth{\rightarrowfill}\kern-1.7ex
                 \hbox to \w@dth{\rightarrowfill}}%
       }\limits
\ifx\n@one\empty\else ^{\box\z@}\fi
\ifx\n@two\empty\else _{\box\@ne}\fi}}
\def\tw@@^#1{\@ifnextchar_ {\two@^{#1}}{\two@^{#1}_{}}}

\def\tofr@^#1_#2{\def\n@one{#1}\def\n@two{#2}\mathrel{\setw@dth{#1}{#2}
\mathop{\vcenter{\hbox to \w@dth{\rightarrowfill}\kern-1.7ex
                 \hbox to \w@dth{\leftarrowfill}}%
       }\limits
\ifx\n@one\empty\else ^{\box\z@}\fi
\ifx\n@two\empty\else _{\box\@ne}\fi}}
\def\t@fr@^#1{\@ifnextchar_ {\tofr@^{#1}}{\tofr@^{#1}_{}}}

\newdimen\W@dth
\def\setW@dth#1#2{\setbox\z@\hbox{$#1$}\W@dth=\wd\z@
\setbox\@ne\hbox{$#2$}\ifnum\W@dth<\wd\@ne \W@dth=\wd\@ne \fi
\advance\W@dth by 1.2em}

\def\T@^#1_#2{\allowbreak\def\N@one{#1}\def\N@two{#2}\mathrel
{\setW@dth{#1}{#2}
\mathop{\hbox to \W@dth{\rightarrowfill}}\limits
\ifx\N@one\empty\else ^{\box\z@}\fi
\ifx\N@two\empty\else _{\box\@ne}\fi}}
\def\T@@^#1{\@ifnextchar_ {\T@^{#1}}{\T@^{#1}_{}}}

\def\T@left^#1_#2{\def\N@one{#1}\def\N@two{#2}\mathrel{\setW@dth{#1}{#2}
\mathop{\hbox to \W@dth{\leftarrowfill}}\limits
\ifx\N@one\empty\else ^{\box\z@}\fi
\ifx\N@two\empty\else _{\box\@ne}\fi}}
\def\T@@left^#1{\@ifnextchar_ {\T@left^{#1}}{\T@left^{#1}_{}}}

\def\Tofr@^#1_#2{\def\N@one{#1}\def\N@two{#2}\mathrel{\setW@dth{#1}{#2}
\mathop{\vcenter{\hbox to \W@dth{\rightarrowfill}\kern-1.7ex
                 \hbox to \W@dth{\leftarrowfill}}%
       }\limits
\ifx\N@one\empty\else ^{\box\z@}\fi
\ifx\N@two\empty\else _{\box\@ne}\fi}}
\def\T@fr@^#1{\@ifnextchar_ {\Tofr@^{#1}}{\Tofr@^{#1}_{}}}

\def\Two@^#1_#2{\def\N@one{#1}\def\N@two{#2}\mathrel{\setW@dth{#1}{#2}
\mathop{\vcenter{\hbox to \W@dth{\rightarrowfill}\kern-1.7ex
                 \hbox to \W@dth{\rightarrowfill}}%
       }\limits
\ifx\N@one\empty\else ^{\box\z@}\fi
\ifx\N@two\empty\else _{\box\@ne}\fi}}
\def\Tw@@^#1{\@ifnextchar_ {\Two@^{#1}}{\Two@^{#1}_{}}}

\def\to{\@ifnextchar^ {\t@@}{\t@@^{}}}
\def\from{\@ifnextchar^ {\t@@left}{\t@@left^{}}}
\def\tofro{\@ifnextchar^ {\t@fr@}{\t@fr@^{}}}
\def\To{\@ifnextchar^ {\T@@}{\T@@^{}}}
\def\From{\@ifnextchar^ {\T@@left}{\T@@left^{}}}
\def\Two{\@ifnextchar^ {\Tw@@}{\Tw@@^{}}}
\def\Tofro{\@ifnextchar^ {\T@fr@}{\T@fr@^{}}}

\makeatother

\title{From Moore-Penrose to Markov via Gauss}
\author{Cole Comfort and Jean-Simon Pacaud Lemay}

\begin{document}
\allowdisplaybreaks




\maketitle

 \begin{abstract}Markov categories are the central framework for categorical probability theory. Many important concepts from probability theory can be formalized in terms of Markov categories. In particular, conditional probability distributions and Bayes’ theorem are captured via the notion of conditionals in a Markov category. Gaussian probability theory gives an example of a Markov category with conditionals, where the conditionals can be computed using the Moore-Penrose inverse. In this paper, we introduce the Gauss construction on a Moore-Penrose dagger additive category, producing a Markov category with conditionals. Applying the Gauss construction to the category of real matrices recaptures the Gaussian probability theory example, while applying it to the category of complex (resp. quaternionic) matrices gives us new Markov categories of proper complex (resp. quaternionic) Gaussian conditional distributions. Moreover, we also characterize all possible conditionals in the Gauss construction. \end{abstract}

 \noindent \small \textbf{Acknowledgements.} The authors would like to thank Tobias Fritz, Dario Stein, Matthew Di Meglio, and Chris Heunen for useful discussions and interest in this research project. This work has been partially funded by the European Union through the
MSCA SE project QCOMICAL and within the framework of ``Plan France
2030'', under the research projects EPIQ ANR-22-PETQ-0007 and HQI-R\&D
ANR-22-PNCQ-0002. This material is also based upon work supported by the AFOSR under award number FA9550-24-1-0008.

 \tableofcontents

\newpage
\section{Introduction}

Categorical probability theory, as the name suggests, is the field of research of applying category theory to help study the underlying structures for probability theory. By using categorical tools to provide a synthetic version of probability theory, one can then generalize models and results to more general settings. While categorical probability theory has been around since the works of Lawvere \cite{lawvere1962category} and Giry \cite{giry2006categorical}, the field itself really took off in the late 2010s with the landmark paper of Cho and Jacobs \cite{cho2019disintegration}. As such, since the beginning of the 2020s, categorical probability theory is now a well-established field and an active area of research and interest. We invite the reader to see Fritz' paper \cite{fritz2020synthetic} for a lovely introduction to categorical probability theory and its interesting history. 

The primary framework for categorical probability theory is Markov categories (Def~\ref{def:Markov}) which are, briefly, symmetric monoidal categories with copy and discard operations, but where only the discard is preserved by maps. Earlier concepts of Markov categories were introduced by Golubtsov \cite{golubtsov2002monoidal} and Fong \cite{fong2012}, but were fully defined by Cho and Jacobs in \cite{cho2019disintegration} under the name affine copy-discard (CD) categories. In \cite{fritz2020synthetic}, Fritz renamed them Markov categories since Markov kernels, a fundamental structure in probability theory, themselves form a Markov category. Of course, Markov categories are named after Russian mathematician Andrey Markov, famous for his foundational works on Markov chains. 

Many important concepts and results from probability theory can be cleanly formulated in terms of Markov categories. Of particular importance is the notion of conditionals in a Markov category \cite[Sec 11]{fritz2020synthetic} which, as the name suggests, captures the notion of a conditional probability distribution. Briefly, given a map of suitable type in a Markov category, a conditional for this map (Def~\ref{def:conditional}), if it exists, is a sort of transpose map which takes one of the outputs of the codomain as an input in the domain. We can also interpret Bayes' theorem, one of the most well-known theorems from probability theory, in a Markov category with conditionals \cite{cho2019disintegration,golubtsov2002monoidal}. 

Having all conditionals is a highly desirable property for a Markov category to have, since conditionals are highly useful and imply various other interesting properties, see \cite{di2025partial,fritz2020synthetic} for more details. There are many interesting examples of Markov categories with conditionals which capture various frameworks for probability theory. In particular, there is a Markov category with conditionals called $\mathsf{Gauss}$ (Ex~\ref{ex:Gauss}) that captures Gaussian probability theory and whose morphisms are real Gaussian conditional distributions \cite[Ex 6]{fritz2020synthetic}. Of course, Gaussian distributions are one of the most important concepts in statistics and probability,  named after the forever famous German mathematician Carl Friedrich Gauss. Moreover, conditionals in this Markov category $\mathsf{Gauss}$ can be computed using the Moore-Penrose inverse. 

For matrices, the Moore-Penrose inverses is a special kind of generalized inverse with many practical applications, see \cite{campbell2009generalized}. While a Moore-Penrose inverse may not exist for matrices over an arbitrary ring, if one does exist,  then it is unique. In particular, all matrices over the real, complex, and quaternion numbers have a Moore-Penrose inverse. This kind of generalized inverse is named after American mathematician E. Hastings Moore, who first introduced the concept in terms of orthogonal projectors \cite{moore1920reciprocal}, and English mathematician and physicist Roger Penrose, who later (without knowing of Moore's work) introduced the concept in terms of algebraic identities \cite{penrose1955generalized}. The notion of Moore-Penrose inverses can be generalized to other algebraic settings with an involution such as, in particular, to dagger categories. 

Briefly, a dagger category (Def~\ref{def:dagger}) is a category with an involution operation on its maps, which for every map produces a map of dual type. Dagger categories have now been extensively studied, especially during the 2000s and 2010s, and are a fundamental concept for categorical foundations of quantum theory, see \cite{heunen2019categories}. In any dagger category one can define the notion of a Moore-Penrose inverse for maps (Def~\ref{def:MP}). Moore-Penrose inverses in dagger categories were first introduced and studied by Puystjens and Robinson in a series of papers in the 1980s, see for example \cite{puystjens1981moore}, though the concept was mostly left untouched for many years. However, with all the significant developments in dagger category theory, Cockett and the second named author picked up the study of Moore-Penrose inverses in dagger categories again in \cite{EPTCS384.10}. This has led to a recent surge of interest in Moore-Penrose inverses in dagger category theory, such as for new applications in categorical quantum foundations \cite{EPTCS426.3}, graphical languages \cite{booth2024graphical}, developing new generalized inverses in category theory \cite{cockett2024drazin,EPTCS426.12}, and studying Moore-Penrose inverses in new areas \cite{auffarth2025pseudoinversos}. 

The main objective of this paper is to generalize the category $\mathsf{Gauss}$, of real Gaussian conditional distributions,  using the theory of Moore-Penrose dagger categories,  providing new examples of Markov categories with conditionals. To do so, we first observe that Gaussian conditional distributions, and thus maps in $\mathsf{Gauss}$, are fully described in terms of real matrices, while composition in $\mathsf{Gauss}$ makes use of the transpose and sum. As such, it is natural generalize this to the setting of dagger additive categories (Sec~\ref{sec:dac}), so dagger categories with compatible biproducts. This is further justified by the fact that maps in an dagger additive category have a canonical matrix representation \cite[Sec 2.2.4]{heunen2019categories}, so composition is given by matrix multiplication and the involution is given by the conjugate transpose of matrices. Therefore, we introduce the Gauss construction (Sec~\ref{sec:gauss}) on an dagger additive category which produces a Markov category (Thm~\ref{thm:gauss-markov}). The main result of this paper is that when we apply the Gauss construction to a Moore-Penrose dagger additive category, we obtain a Markov category with conditionals (Thm~\ref{thm:MP-Gauss-Markov}). Applying the Gauss construction to real matrices recaptures $\mathsf{Gauss}$ (Ex~\ref{ex:Gauss-real-mat}), while applying the Gauss construction to complex matrices gives a new Markov category of proper complex Gaussian conditional distributions (Ex~\ref{ex:GaussC}).  Similarly, applying the Gauss construction to quaternionic matrices gives a new Markov category of proper quaternionic Gaussian conditional distributions (Ex~\ref{ex:GaussH}).  

An important fact about conditionals that we have yet to mention is that they are not necessarily unique, even if a Markov category has all conditionals. So a map in a Markov category can have multiple possible conditionals. As such, the Moore-Penrose inverse gives but one possible way of building a conditional in the Gauss construction. We characterize all possible conditionals in the Gauss construction (Sec~\ref{sec:cond-in-gauss}) and show that they can be fully described by special maps in the base dagger additive category which we call conditional generators (Def~\ref{def:conditional-generator}). We then show that one can use a weaker version of the Moore-Penrose inverse (Def~\ref{def:MP-ijk}), one that generalizes least squares inverses \cite[Fig 6.1]{campbell2009generalized}, to also build conditionals in the Gauss construction (Lemma~\ref{lemma:phi-bullet}). That said, since (when they exist) Moore-Penrose inverses are always unique, they give a canonical choice of conditionals in the Gauss construction. 

\textbf{Acknowledgements:} The authors would like to thank Tobias Fritz, Dario Stein, Matthew Di Meglio, and Chris Heunen for useful discussions and interest in this research project. 

\section{Markov Categories with Conditionals}

In this background section, we briefly review the basics of the theory of Markov categories, including the notion of conditionals. For an in-depth introduction to Markov categories and categorial probability theory, we refer the reader to~\cite{fritz2020synthetic}. 

We assume that the reader is familiar with the basics of category theory and with symmetric monoidal categories. Briefly, a Markov category is a symmetric monoidal category where each object is a cocommutative comonoid (the dual notion of a commutative monoid) in a coherent way, such that all maps preserve the counit, but importantly not necessarily the comultiplication. For simplicity, in this paper we will work with symmetric \textit{strict} monoidal categories (which is not an issue for Markov category per \cite[Thm 10.17]{fritz2020synthetic}), so we will assume that the associativity and unital isomorphisms for the monoidal product are identities. An arbitrary category will be denoted by $\mathbb{X}$, with objects denoted by capital letters such as $A,B, C$ or $X,Y,Z$, and maps denoted by miniscule letters such as $f,g,h$ or $p,q,s$. Homsets will be denoted as $\mathbb{X}(A,B)$ and maps $f \in \mathbb{X}(A,B)$ as arrows $f: A \to B$. Composition will be denoted using the usual $\circ$, while identity maps by $\mathsf{id}_A: A \to A$. An arbitrary symmetric monoidal category will be denoted with underlying category $\mathbb{X}$, monoidal product $\otimes$, monoidal unit $I$, and symmetry natural isomorphism $\mathsf{swap}_{A,B} = A \otimes B \to B \otimes A$. By strictness, we may write $A_1 \otimes \hdots \otimes A_n$ and $A \otimes I = A = A \otimes I$. We will denote symmetric monoidal categories simply by their underlying category $\mathbb{X}$. For an in-depth introduction to symmetric monoidal categories and their internal comonoids, we refer the reader to \cite{heunen2019categories}.  

\begin{defi}\label{def:Markov} \cite[Def 2.1]{fritz2020synthetic} A \textbf{Markov category} is a symmetric monoidal category $\mathbb{X}$ equipped with two families of maps (indexed by objects $A$) $\mathsf{copy}_A: A \to A \otimes A$, called the \textbf{copy} or \textbf{comultiplication}, and $\mathsf{del}_A: A \to I$, called the \textbf{delete} or \textbf{counit}, such that: 
\begin{enumerate}[{\em (i)}]
\item For each object $A$, $(A, \mathsf{copy}_A, \mathsf{del}_A)$ is a \textbf{cocommutative comonoid}, that is, the following equalities hold:
\begin{gather*}
(\mathsf{copy}_A \otimes \mathsf{id}_A) \circ \mathsf{copy}_A = (\mathsf{id}_A \otimes \mathsf{copy}_A) \circ \mathsf{copy}_A \\
(\mathsf{del}_A \otimes \mathsf{id}_A) \circ \mathsf{copy}_A = \mathsf{id}_A = (\mathsf{id}_A \otimes \mathsf{del}_A) \circ \mathsf{copy}_A \\
\mathsf{swap}_{A,A} \circ \mathsf{copy}_A = \mathsf{copy}_A
\end{gather*}
\item $\mathsf{copy}$ and $\mathsf{del}$ are coherent with the monoidal structure, that is, the following equalities hold: 
\begin{gather*}
\mathsf{copy}_{A \otimes B} = (\mathsf{id}_A \otimes \mathsf{swap}_{A,B} \otimes \mathsf{id}_B) \circ (\mathsf{copy}_A \otimes \mathsf{copy}_B) \\
\mathsf{del}_{A \otimes B} = \mathsf{del}_A \otimes \mathsf{del}_B
\end{gather*}
\item $\mathsf{del}$ is a natural transformation, that is, for every map $f: A \to B$, the following equality holds: 
\begin{align*}
\mathsf{del}_B \circ f = \mathsf{del}_A 
\end{align*}
 \end{enumerate}
\end{defi}

Readers familiar with comonoids will note that asking for $\mathsf{copy}$ and $\mathsf{del}$ be coherent with the monoidal structure is precisely asking that the comonoid structure on $A \otimes B$ be the monoidal product of the comonoid $A$ and $B$ in the usual sense \cite[Lemma 4.8]{heunen2019categories}. Moreover, asking that $\mathsf{del}$ be a natural transformation is equivalent to saying that the monoidal unit $I$ is a \textbf{terminal object} \cite[Prop 4.15]{heunen2019categories}, which  means that there is a unique map from every object into $I$. In a Markov categor, this unique map is precisely $\mathsf{del}_A$. It is again worth stressing that while maps preserve the counit, maps need not preserve the comultiplication, in other words, we do not ask that $\mathsf{copy}$ be natural. Maps that do preserve the comultiplication are called \textit{deterministic}. 

\begin{defi} \cite[Def 10.1]{fritz2020synthetic} In a Markov category $\mathbb{X}$, a map $f: A \to B$ is \textbf{deterministic} if the following equality holds:
\begin{align*} 
\mathsf{copy}_B \circ f = (f \otimes f) \circ \mathsf{copy}_A 
\end{align*}
We denote the subcategory of deterministic maps by $\mathsf{DET}[\mathbb{X}]$. 
\end{defi}

The deterministic maps of a Markov category form a subcategory, and since $\mathsf{copy}$ and $\mathsf{del}$ are deterministic, the deterministic maps actually form a sub-Markov category \cite[Lemma 10.12]{fritz2020synthetic}. Moreover, the subcategory of deterministic maps is in fact a \textit{Cartesian} monoidal category \cite[Remark 10.13]{fritz2020synthetic}, that is, a symmetric monoidal category whose monoidal product is an actual product and whose monoidal unit is a terminal object. 

At this point, it may be useful to briefly give some probability theory interpretations for Markov categories -- a more in-depth discussion of these interpretations can be found in~\cite{fritz2020synthetic}. Maps $f: A \to B$ in a Markov category are interpreted as \textit{Markov kernels}, so ``a function with random outcomes,''  basically means that $f$ assigns to every input value of type $A$ a probability distribution over output values of type $B$. Now $\mathsf{del}_A: A \to I$ simply discards the input and produces no output, while $\mathsf{copy}_A: A \to A \otimes A$ takes in an input value and outputs two copies of the input, without introducing any randomness. In general, applying a Markov kernel $f$ independently to two copies of an input is not equal to applying the Markov kernel first and then copying the result. Thus deterministic maps are those whose random outputs are independent of themselves, thus we can copy them. 

A highly desirable property of a Markov category is having \textit{conditionals}. As the name suggests, having conditionals formalizes the notion of conditional probability distributions, allowing for an interpretation of the famous Bayes' theorem in a Markov category. 

\begin{defi}\label{def:conditional} \cite[Def 11.5]{fritz2020synthetic} In a Markov category $\mathbb{X}$, a \textbf{conditional} for a map $f: A \to B \otimes C$ is a map of type $f\vert_B: B \otimes A \to C$ such that the following equality holds: 
\begin{align}\label{eq:conditional}
(\mathsf{id}_B \otimes f\vert_{B}) \!\circ\! (\mathsf{copy}_B \otimes \mathsf{id}_A) \circ (\mathsf{id}_B \otimes \mathsf{del}_C \otimes \mathsf{id}_A) \circ (f \otimes \mathsf{id}_A) \circ \mathsf{copy}_A \!= \! f
\end{align}
A \textbf{Markov category with conditionals}\footnote{Some references include having conditionals as part of the definition of a Markov category, such as in \cite[Def 2.4]{di2025partial}} is a Markov category such that for all objects $A$, $B$, and $C$, every map $f: A \to B \otimes C$ has a conditional. 
\end{defi}

Of course, for a map $f: A \to B \otimes C$, the above definition could be called a \textit{left} conditional, so we could symmetrically instead ask for a \textit{right} conditional $f\vert_C: A \otimes C \to B$, which would satisfy: 
\begin{align}\label{eq:conditional-right}
(f\vert_{C} \otimes \mathsf{id}_C) \!\circ\! (\mathsf{id}_A \otimes \mathsf{copy}_C) \circ (\mathsf{id}_A \otimes \mathsf{del}_B \otimes \mathsf{id}_C) \circ (\mathsf{id}_A \otimes f) \circ \mathsf{copy}_A \!=\! f
\end{align}
However it is not difficult to see that a right condition of $f$ is precisely a left conditional of $\mathsf{swap}_{B,C} \circ f: A \to C \otimes B$ precomposed with $\mathsf{swap}_{A,C}$. Thus having all left (resp. right) conditionals implies having all right (resp. left) conditionals. Throughout this paper, we will mostly deal with left conditionals and refer to them simply as conditionals. 

It is important to stress that conditionals are \textit{not} unique. Indeed given Markov category with conditionals, the conditionals are unique if and only if the base category is posetal \cite[Prop 11.15]{fritz2020synthetic}, that is, when all parallel maps are equal, or in other words, when there is at most one map between each object. Instead, conditionals are unique up to almost surely equality \cite[Prop 13.7]{fritz2020synthetic}. We will not review almost surely equality here, and invite the curious reader to see \cite{cho2019disintegration,fritz2020synthetic}. 

To help understand (\ref{eq:conditional}), it will be useful to introduce some useful notation for intuition. Following \cite[Notation 2.8]{fritz2020synthetic}, and borrowing from probability theory, given a map $f: A_1 \otimes \hdots \otimes A_n \to B_1 \otimes \hdots \otimes B_m$, we interpret  the $A_i$ as known outcomes, whereas we interpret the $B_j$ as jointly random variables/outcomes. So we suggestively write $f(b_1, \hdots, b_m \vert a_1, \hdots, a_n)$ to represent the probability of outcome $(b_1, \hdots, b_m) \in B_1 \otimes \hdots \otimes B_m$ given that $(a_1, \hdots, a_n) \in A_1 \otimes \hdots \otimes A_n$ is true. Then given a map $f: A \to B \otimes C$, its conditional $f\vert_B: B \otimes A \to C$ is interpreted as the probability distribution of $C$ when the outcome of $B$ is known. Thus (\ref{eq:conditional}) says that $f(b,c\vert a) = f(b\vert a)  f\vert_B(c \vert b,a)$, which means that probability of outcome $(b,c)$ if $a$ is true is equal to the probability of outcome $b$ if $a$ is true (forgetting about $c$) multiplied by the probability outcome of $c$ if $a$ is true knowing $b$. Now if $f$ also has a right conditional $f\vert_C$, then by combining (\ref{eq:conditional}) and (\ref{eq:conditional-right}), we also get that $f\vert_B(b \vert c,a)f(c\vert a) = f(b\vert a)  f\vert_C(c \vert b,a)$. In the special case of setting $A=I$, we may rewrite this latest equality as $f\vert_B(b \vert c)f(c) = f(b) f\vert_C(c\vert b)$, and if $f(c) \neq 0$, we recapture the famous Bayes' theorem $f\vert_B(b \vert c) = \frac{f(b) f\vert_C(c\vert b)}{f(c)}$. For more discussions on conditionals in Markov categories, we invite the reader to see \cite{di2025partial,di2025order,fritz2020synthetic}. 

Now there are many interesting examples of Markov categories with conditionals, many of which capture various important aspects of probability theory, such as the category of finite sets and Markov kernels \cite[Ex 2.5]{fritz2020synthetic}. The main example of interest for the story of this paper is the Markov category which captures \textit{Gaussian probability theory}. 

\begin{example}\label{ex:Gauss} Let $\mathbb{R}$ be the field of real numbers. Define $\mathsf{Gauss}$ \cite[Ex 6]{fritz2020synthetic} to be the category whose objects are the natural numbers $n \in \mathbb{N}$ and where a map $(M,C,s): n \to m$ is triple consisting of  an $m \times n$ $\mathbb{R}$-matrix $M$, a positive semidefinite square $m \times m$ $\mathbb{R}$-matrix $C$, and a (column) vector $s \in \mathbb{R}^m$. A map $(M,C,s): n \to m$ represents the Gaussian conditional distribution of a random variable $Y \in \mathbb{R}^m$ in terms of $X \in \mathbb{R}^n$ as $Y = MX + \xi$, where $\xi \in \mathbb{R}^m$ is Gaussian noise independent from $X$, with expected value $\mathsf{E}[\xi] = s$ and covariance  $\mathsf{VAR}(\xi) = \mathsf{E}[(\xi - E[\xi])(\xi - E[\xi])^{\mathsf{T}}]= C$. Identity maps are given by triples $\mathsf{id}_n := (I_n, 0,0): n \to n$, where $I_n$ is the $n\times n$ identity matrix. While composition is given by:
\[(N,D,t) \circ (M,C,s) = (NM, NCN^{\mathsf{T}} + D, Ns + t)\]
where $\mathsf{T}$ is the transpose operator. Composition corresponds to the \textit{direct combination} of the corresponding Gaussian conditional distributions \cite[Sec 4.4]{lauritzen2001stable}. $\mathsf{Gauss}$ is also a symmetric monoidal category where the monoidal unit is $I =0$, the monoidal product is defined on objects as $n \otimes m= n+m$ and on maps as
\begin{align*}
(M,C,s) \otimes (N,D,t) = \left( \begin{bmatrix} M & 0 \\ 
0 & N \end{bmatrix}, \begin{bmatrix} C & 0 \\ 
0 & D \end{bmatrix}, \begin{bmatrix} s \\ t
\end{bmatrix} \right) 
\end{align*}
 and where the symmetry is given by the triple
\[
\mathsf{swap}_{n,m} \coloneqq \left( \begin{bmatrix} 0 & I_m \\ I_n & 0  \end{bmatrix} ,0,0 \right)
\]
$\mathsf{Gauss}$ is a Markov category where the copy and delete are defined as follows: 
\begin{align*}
    \mathsf{copy}_n := \left( \begin{bmatrix} I_n \\ I_n \end{bmatrix}, 0, 0 \right): n \to n + n && \mathsf{del}_n := (0,0,0): n \to 0
\end{align*}
The deterministic maps are precisely the triples of the form $(M,0,s): n \to m$ \cite[Ex 10.7]{fritz2020synthetic}, that is, the Gaussian conditional distributions whose covariance is zero, meaning that no randomness is involved. $\mathsf{Gauss}$ also has conditionals \cite[Ex 11.8]{fritz2020synthetic}. In particular, for a map of type $(M, C, s): n \to m+k$, we may decompose it in block notation as follows: 
\[ (M, C, s) = \left( \begin{bmatrix} M_1 \\ M_2 \end{bmatrix}, \begin{bmatrix} C_{1} &  C_2 \\ C_3 & C_4 \end{bmatrix}, \begin{bmatrix} s_1 \\ s_2 \end{bmatrix} \right) \]
where $M_1$ is an $m \times n$ matrix, $M_2$ an $k \times n$ matrix, $C_1$ an $m \times m$ matrix, $C_2$ an $m \times k$ matrix, $C_3$ an $k \times m$ matrix, $C_4$ an $k \times k$ matrix, $s_1$ a column vector of size $m$, and lastly $s_2$ is a column vector of size $k$. Then the following map of type $m +n \to k$ is a conditional for $(M,C,s)$
\[ \left( \begin{bmatrix} C_3C_1^\circ & M_2 - C_3 C_1^\circ M_1 \end{bmatrix}, C_4 - C_3 C_1^\circ C_2, s_2 - C_3 C_1^\circ s_1 \right) \]
where $C_1^\circ$ is the \emph{Moore-Penrose inverse} of $C_1$. Recall that Moore-Penrose inverses are special kinds of generalized inverses that exist for any real matrix. We will review Moore-Penrose inverses in Sec~\ref{sec:MP} below. However, since conditionals are not necessarily unique, the Moore-Penrose inverses allows us to build one possible conditional but there may be others. We will investigate this further in Sec~\ref{sec:cond-in-gauss}. 
\end{example}

\section{Dagger Additive Categories and their Matrix Representation}\label{sec:dac}

The main objective of this paper is to generalize the Markov category of Gaussian probability theory by replacing matrices with maps of a \textit{dagger additive category}. So in this background section, we review dagger additive categories and, in particular, their matrix representation. For an in-depth introduction to dagger (additive) categories, we refer the reader to \cite{heunen2019categories}. 

\begin{defi}\label{def:dagger} \cite[Def 2.34]{heunen2019categories} A \textbf{dagger} on a category $\mathbb{X}$ is a contravariant functor $(\_)^\dagger: \mathbb{X} \to \mathbb{X}$ which is the identity on objects and involutive. More concretely, a dagger $\dagger$ associates each map ${f: A \to B}$ to a chosen map of dual type $f^\dagger: B \to A$, called the \textbf{adjoint} of $f$, such that the following equalities hold: 
\begin{align}
\mathsf{id}_A^\dagger = \mathsf{id}_A && (g \circ f)^\dagger = f^\dagger \circ g^\dagger && (f^\dagger)^\dagger = f
\end{align}
A \textbf{dagger category} is a pair $(\mathbb{X}, \dagger)$ consisting of a category $\mathbb{X}$ equipped with a dagger $\dagger$. 
\end{defi}

It is important to note that a category can have multiple different daggers. This means that a dagger on a category is structure which must be chosen. Oftentimes, when it is clear from context, dagger categories are referred to by their underlying category. However in this paper, we will encounter categories with multiple daggers, so we will always write dagger categories as pairs. This will be particularly important in Sec~\ref{sec:MP}, where a category can have Moore-Penrose inverses with respect to one dagger but not with respect to another. 

We now add additive structure to our dagger categories. First recall that a \textbf{pre-additive category} \cite[Def 1.2.1]{borceux1994handbook} is a category enriched over abelian groups, that is, a category $\mathbb{X}$ such that each homset $\mathbb{X}(A,B)$ is an abelian group, with binary operation $+: \mathbb{X}(A,B) \times \mathbb{X}(A,B) \to \mathbb{X}(A,B)$, unit $0 \in \mathbb{X}(A,B)$, and inverse $-: \mathbb{X}(A,B) \to \mathbb{X}(A,B)$, such that composition is a group morphism, that is, the following equalities hold: 
\begin{align*}
g \circ (f_1 + f_2) \circ h = g \circ f_1 \circ h + g \circ f_2 \circ h && g \circ 0 \circ f = 0 
\end{align*}
For dagger categories, we also ask that the dagger be a group homomorphism. 

\begin{defi} A \textbf{dagger pre-additive category} is a dagger category $(\mathbb{X}, \dagger)$ such that $\mathbb{X}$ is a pre-additive category, where additionally $\dagger$ is a group homomorphism, that is, the following equalities hold: 
\begin{align*}
(f+g)^\dagger = f^\dagger + g^\dagger && 0^\dagger = 0 && (-f)^\dagger = -f^\dagger 
\end{align*}
\end{defi}

We can also add biproducts to our dagger categories. First recall that a \textbf{zero object} \cite[Def 1.1.1]{borceux1994handbook} in a category is an object that is both terminal and initial. In a pre-additive category, a zero object can equivalently be defined as an object $\mathsf{0}$ such that for every object $A$, the zero maps $0: A \to \mathsf{0}$ and $0: \mathsf{0} \to A$ are the unique maps to and from $\mathsf{0}$ \cite[Prop 1.2.3]{borceux1994handbook}. Next recall that a \textbf{biproduct} is both a product and coproduct satisfying certain compatibility relations. In a pre-additive category, the \textbf{biproduct} \cite[Def 1.1.1]{borceux1994handbook} of a finite family of objects $A_1$,..., $A_n$, is an object $A_1 \oplus \hdots A_n$ with maps (for all $1 \leq j \leq n$) $\pi_j: A_1 \oplus \hdots A_n \to A_j$, called the \textbf{projections}, and $\iota_j: A_j \to A_1 \oplus \hdots A_n$, called the \textbf{injections}, such that the following equalities hold:
\begin{align*}
\pi_j \circ \iota_i = \begin{cases} \mathsf{id}_{A_j} & \text{if $i=j$}\\
0 & \text{if $i \neq j$}\end{cases} && \iota_1 \circ \pi_1 + \hdots + \iota_n \circ \pi_n = \mathsf{id}_{A_1 \oplus \hdots \oplus A_n} 
\end{align*}
The biproduct of the empty family, if it exists, is precisely a zero object. Then an \textbf{additive category} \cite[Def 1.2.6]{borceux1994handbook} is a pre-additive category with all finite biproducts. For dagger categories, \textit{dagger biproducts} are biproducts whose projections and injections are adjoints of each other.  

\begin{defi}\cite[Def 2.39]{heunen2019categories} A \textbf{dagger additive category} is a dagger pre-additive category $(\mathbb{X}, \dagger)$ which has all finite \textbf{$\dagger$-biproducts}, that is, it has all finite biproducts such that the following equality holds (for all finite families of objects and all $j \in \mathbb{N}$): 
\begin{align*}
\pi_j^\dagger = \iota_j 
\end{align*}
\end{defi}

There are many interesting examples of dagger additive categories, such as the category of Hilbert spaces \cite[Def 2.31]{heunen2019categories} or the category of sets and relations \cite[Def 2.33]{heunen2019categories}. Our main examples of interest in this paper are the dagger additive categories of matrices over a ring.  

\begin{example} Let $R$ be a (unital and associative, but not necessarily commutative) ring. Define $\mathsf{MAT}(R)$ to be the category of $R$-matrices, that is, the category whose objects are the natural numbers $n \in \mathbb{N}$ and where a map ${A: n \to m}$ is an $m \times n$ $R$-matrix $A$, whose coefficients we denote by $A(i,j)$ for $1\leq i \leq m$ and $1 \leq j \leq n$. The identity map of $n$ is the $n \times n$ identity matrix, while composition is given by matrix multiplication, so $A \circ B = AB$. For an $m \times n$ matrix $A$ (which is a map of type $n \to m)$, let $A^\mathsf{T}$ be its transpose, that is, the $n \times m$ matrix (so a map of dual type $m \to n$) given by $A^\mathsf{T}(i,j) = A(j,i)$. Then this makes $(\mathsf{MAT}(R), \mathsf{T})$ into a dagger category. Moreover, $(\mathsf{MAT}(R), \mathsf{T})$ is a dagger additive category where the sum of maps is given by the usual sum of matrices, $(A+B)(i,j) = A(i,j) + B(i,j)$, the zero map is the zero matrix, $0(i,j) = 0$, the negative is the usual negation of a matrix, $(-A)(i,j) = -A(i,j)$, and where the biproduct on objects is given by taking the sum, $n_1 \oplus \hdots \oplus n_m = n_1 + \hdots + n_m$. Our main example will be when $R = \mathbb{R}$, the field of real numbers, giving us the dagger category of real matrices $(\mathsf{MAT}(\mathbb{R}), \mathsf{T})$. 
\end{example}

If the base ring is an involutive ring, then its category of matrices admits another dagger given by conjugate transpose.

\begin{example} Recall that an involutive ring is a ring $R$ equipped with a unary operation $\overline{(\_)}:R\to R$, called the involution, such that $\overline{x+y} = \overline{x} + \overline{y}$, $\overline{1} = 1$, $\overline{xy} = \overline{y}~ \overline{x}$, and $\overline{\overline{x}} = x$. Given an $m \times n$ $R$-matrix $A$, let $A^\ast$ denote its conjugate transpose, that is, the $n \times m$ matrix given by $A^\ast(i,j) = \overline{A(j,i)}$. This makes $(\mathsf{MAT}(R), \ast)$ into a dagger category. Moreover, $(\mathsf{MAT}(R), \ast)$ is also dagger additive category with the same additive structure as in the above example. Our two main examples will be when $R = \mathbb{C}$, the field of complex numbers, whose involution is given by complex conjugation, giving us the dagger of complex matrices $(\mathsf{MAT}(\mathbb{C}), \ast)$, or when $R = \mathbb{H}$, the skew field of quaternions, whose involution is given by quaternionic conjugation, giving us the dagger of qauternionic matrices $(\mathsf{MAT}(\mathbb{H}), \ast)$. 
\end{example}

A very practical fact about additive categories is that maps have \textit{matrix representation} \cite[Sec 2.2.4]{heunen2019categories}, which comes from the fact that a biproduct admits the universal properties of both a product and coproduct. In particular, this matrix representation implies that composition in an additive category is essentially given by matrix multiplication. So in an additive category, every map $f: A_1 \oplus \hdots \oplus A_n \to B_1 \oplus \hdots \oplus B_m$ is uniquely determined by a family of maps $f(i,j): A_j \to B_i$, defined below on the left, such that the equality on the right holds: 
\begin{align*} 
f(i,j) = \pi_i \circ f \circ \iota_j && f = \sum_{i,j} \iota_i \circ f(i,j) \circ \pi_j 
\end{align*}
Therefore $f$ can be represented as a $m\times n$ matrix \cite[Lemma 2.26]{heunen2019categories} where the term in the $i$-th row and $j$-th column is $f(i,j)$, so:
\begin{align*}
f = \begin{bmatrix} f(i,j) \end{bmatrix} = \begin{bmatrix} f(1,1) & f(1,2) & \hdots & f(1,n) \\
f(2,1) & f(2,2) & \hdots& f(2,n) \\
\vdots & \vdots & \ddots & \vdots \\
f(m,1) & f(m,2) & \hdots & f(m,n) \end{bmatrix}
\end{align*}
Moreover, two maps in an additive category are equal if and only if they have the same matrix representation \cite[Cor 2.27]{heunen2019categories}. Now given composable maps $f: \bigoplus\limits^n_{i=1} A_i \to \bigoplus\limits^m_{j=1} B_j$ and $g: \bigoplus\limits^m_{j=1} B_j \to \bigoplus\limits^p_{k=1} C_k$, the matrix representation of their composition $g \circ f: \bigoplus\limits^n_{i=1} A_i \to \bigoplus\limits^p_{k=1} C_k$ is given by matrix multiplication \cite[Prop 2.28]{heunen2019categories}, 
\begin{align*}
g \circ f = \begin{bmatrix} g(k,l) \end{bmatrix} \circ \begin{bmatrix} f(i,j) \end{bmatrix} = \begin{bmatrix} \sum \limits^{n}_{j=1} g(j,k) \circ f(i,j) \end{bmatrix} 
\end{align*}
while the matrix representation of the identity map $\mathsf{id}_{A_1 \oplus \hdots \oplus A_m}$ is precisely the identity matrix, that is, with identities on the diagonal and zero everywhere else: 
\begin{align*}
\mathsf{id}_{A_1 \oplus \hdots \oplus A_m} = \begin{bmatrix} \mathsf{id}_{A_1} & 0 & \hdots & 0 \\
0 & \mathsf{id}_{A_2} & \hdots & 0 \\
\vdots & \vdots & \ddots & \vdots \\
0 & 0 & \hdots & \mathsf{id}_{A_m} \end{bmatrix}
\end{align*}
Now for maps $f:  \bigoplus\limits^n_{i=1} A_i \to \bigoplus\limits^m_{j=1} B_j$ and $g:  \bigoplus\limits^n_{i=1} A_i \to \bigoplus\limits^m_{j=1} B_j$, the matrix representation of their sum $f+g$ is given by the usual sum of matrices, so by adding coefficients point-wise, the matrix representation of the negation $-f$ is also by taking the negation pointwise, while the matrix representation of the zero map $0: \bigoplus\limits^n_{i=1} A_i \to \bigoplus\limits^m_{j=1} B_j$ is of course simply the zero matrix: 
\begin{gather*}
f + g = \begin{bmatrix} f(i,j) \end{bmatrix} + \begin{bmatrix} g(i,j) \end{bmatrix} = \begin{bmatrix} f(i,j) + g(i,j) \end{bmatrix} \\
-f = - \begin{bmatrix} f(i,j) \end{bmatrix}  = \begin{bmatrix} -f(i,j) \end{bmatrix} \\
0 = \begin{bmatrix} 0 & 0 & \hdots & 0 \\
0 & 0 & \hdots & 0 \\
\vdots & \vdots & \ddots & \vdots \\
0 & 0 & \hdots & 0 \end{bmatrix}
\end{gather*}
In a dagger additive category, for a map $f: A_1 \oplus \hdots \oplus A_m \to B_1 \oplus \hdots \oplus B_n$, the matrix representation of its adjoints $f^\dagger:  B_1 \oplus \hdots \oplus B_n \to A_1 \oplus \hdots \oplus A_m$ is the conjugate transpose of the matrix representation of $f$ \cite[Lemma 2.41]{heunen2019categories}:
\begin{align*}
f^\dagger = \begin{bmatrix} f(1,1)^\dagger & f(2,1)^\dagger & \hdots & f(n,1)^\dagger \\
f(1,2)^\dagger & f(2,2)^\dagger & \hdots& f(n,2)^\dagger \\
\vdots & \vdots & \ddots & \vdots \\
f(1,m)^\dagger & f(2,m)^\dagger & \hdots & f(n,m)^\dagger \end{bmatrix} 
\end{align*}

\begin{example}  For any ring $R$, notice that in $\mathsf{MAT}(R)$, every object $n$ is the biproduct of $n$-copies of $1$, that is, $n = 1 \oplus \hdots \oplus 1$. Then for a map $A: n \to m$, which is an $m \times n$ matrix $A$, its matrix representation in the additive category sense will be precisely $A$, where in particular $\pi_i \circ A \circ \iota_j$ is precisely $A(i,j)$, the coefficient of $A$ in the $i$-th row and $j$-th column, seen as a map $1 \to 1$ (which is just an element of $R$ seen as a $1 \times 1$ matrix). 
\end{example}

\section{The Gauss Construction}\label{sec:gauss}

In this section, we introduce the \textit{Gauss construction} on a dagger additive category which produces a Markov category. The Gauss construction generalizes the Gaussian probability theory example in the sense that when applying the Gauss construction to the category of real matrices, we recover the Markov category $\mathsf{Gauss}$ from Ex~\ref{ex:Gauss}. 

Before we jump into giving the definition of the Gauss construction, we need to discuss what the maps might look like. Taking $\mathsf{Gauss}$ as our guiding light, it is clear that the objects in the Gauss construction will be the same as the objects of our base dagger additive category, while maps in the Gauss construction will be triples. Now the first component of the triple will be a map in the base category. So we need to address the other two components. In $\mathsf{Gauss}$, recall that the second component was a positive semidefinite square matrix. Now square matrices will correspond to endomorphisms, but what about positive semidefiniteness? It turns out that positive semidefiniteness is captured in a dagger category via the notion of a \textit{positive} map, which is an endomorphism that is the composite of another map and its adjoints.  

\begin{defi} \cite[Def 2.34]{heunen2019categories} In a dagger category $(\mathbb{X}, \dagger)$, a \textbf{$\dagger$-positive map} is an endomorphism $p: A \to A$ such that there exists a map $\phi: A \to B$ satisfying $p = \phi^\dagger \circ \phi$. 
\end{defi}

It is important to stress that for a $\dagger$-positive map $p$ there may be distinct maps $\phi$ such that $p = \phi^\dagger \circ \phi$. 

\begin{example} Recall the well-known fact that a real square $n \times n$ matrix $A$ is positive semidefinite if and only if there exists a real (resp. complex) $m \times n$ matrix $B$ such that $A= B^\mathsf{T}B$. Therefore, in $(\mathsf{MAT}(\mathbb{R}), \mathsf{T})$, the $\mathsf{T}$-positive maps correspond precisely to real (resp. complex) positive semidefinite square matrices. The same is true for complex (resp. quaternionic) matrices, thus $\ast$-positive maps in $(\mathsf{MAT}(\mathbb{C}), \ast)$ (resp. $(\mathsf{MAT}(\mathbb{H}), \ast)$), meaning this time that $A=B^\ast B$, correspond precisely to complex (resp. quaternionic) positive semidefinite square matrices. 
\end{example}

Recall that the third component of a map in $\mathsf{Gauss}$ was a column vector, which we may view as a map of type $1 \to m$ in our category of matrices. While $1$ is indeed a special object in our category of matrices, an arbitrary dagger additive category does not necessarily have such an object. Instead, we need to fix an object of our base category and take maps out of this object for our third components of maps in the Gauss construction. 

We are now in a position to give our Gauss construction. So for the remainder of this section, let $(\mathbb{X}, \dagger)$ be a dagger additive category and fix an object $X \in \mathbb{X}$. Then define $\mathfrak{G}\left[ (\mathbb{X}, \dagger) \right]_X$ to be the category where:
 \begin{enumerate}[{\em (i)}]
\item The objects of $\mathfrak{G}\left[ (\mathbb{X}, \dagger) \right]_X$ are the same objects of $\mathbb{X}$; 
\item A map $F: A \to B$ in $\mathfrak{G}\left[ (\mathbb{X}, \dagger) \right]_X$ is a triple $F=(f, p, x): A \to B$ consisting of a map $f: A \to B$, a $\dagger$-positive map $p: B \to B$, and a map $x: X \to B$;
\item The identity of $A$ in $\mathfrak{G}\left[ (\mathbb{X}, \dagger) \right]_X$ is the triple $\mathsf{Id}_A = (\mathsf{id}_A, 0, 0): A \to A$;
\item Composition of $F = (f,p,x): A \to B$ and $G = (g,q,y): B \to C$ is defined as follows:
\begin{align*}
G \circ F = (g,q,y) \circ (f,p,x) = (g \circ f, q + g \circ p \circ g^\dagger, y + g \circ x): A \to C
\end{align*}
 \end{enumerate}

\begin{lemma} $\mathfrak{G}\left[ (\mathbb{X}, \dagger) \right]_X$ is category. 
\end{lemma}
\begin{proof} We first need to explain why composition and identities are well-defined, in particular, we need to address the second component and explain why it is $\dagger$-positive. Now clearly every zero map $0$ is $\dagger$-positive since $0 = 0^\dagger \circ 0$, and so identity maps are well-defined in $\mathfrak{G}\left[ (\mathbb{X}, \dagger) \right]_X$. Now for composition, we need to justify why the sum of $\dagger$-positive maps is again $\dagger$-positive. In a general dagger pre-additive category, this need not be the case! That said, in a dagger additive category, dagger biproducts come to the rescue, and in this case the sum of $\dagger$-positive maps is indeed $\dagger$-positive. To see this, consider two $\dagger$-positive maps $\phi: A \to A$ and $\psi: A \to A$, where $\phi = \alpha^\dagger \circ \alpha$ for some $\alpha: A \to B$ and $\psi = \beta^\dagger \circ \beta$ for some ${\beta: A \to C}$. So consider the map $\begin{bmatrix} \alpha \\ \beta  \end{bmatrix}: A \to B \oplus C$, and we compute that: 
\[ \begin{bmatrix} \alpha \\ \beta \end{bmatrix}^\dagger \circ \begin{bmatrix} \alpha \\ \beta \end{bmatrix} = \begin{bmatrix} \alpha^\dagger & \beta^\dagger \end{bmatrix} \circ \begin{bmatrix} \alpha \\ \beta \end{bmatrix} =  \alpha^\dagger \circ \alpha + \beta^\dagger \circ \beta = \phi + \psi \]
So in the presence of $\dagger$-biproducts, the sum of $\dagger$-positive maps is again $\dagger$-positive. As such, composition in $\mathfrak{G}\left[ (\mathbb{X}, \dagger) \right]_X$ is well-defined. Now it is easy to see that composition in $\mathfrak{G}\left[ (\mathbb{X}, \dagger) \right]_X$ is associative and unital since composition and addition is associative and unital in $\mathbb{X}$, $\dagger$ is a contravariant functor, and also that composition and $\dagger$ preserves the additive structure. So we conclude that $\mathfrak{G}\left[ (\mathbb{X}, \dagger) \right]_X$ is indeed a category. 
\end{proof}

Next we define a monoidal structure on  $\mathfrak{G}\left[ (\mathbb{X}, \dagger) \right]_X$. To do so, recall that every additive category is a symmetric monoidal category whose monoidal structure is given its finite biproduct. Explicitly, the monoidal product is the biproduct $\oplus$, which is defined on object as simply $A \oplus B$, while for maps $f: A \to B$ and $g: C \to D$, their biproduct $f \oplus g: A \oplus C \to B \oplus D$ is expressed as a diagonal matrix: 
\[ f \oplus g = \begin{bmatrix} f & 0 \\ 0 & g \end{bmatrix}  \]
The monoidal unit is the zero object $\mathsf{0}$, while the natural symmetry isomorphism $\mathsf{swap}^\oplus_{A,B}: A \oplus B \to B \oplus A$ is the $2 \times 2$ permutation swap matrix: 
\[ \mathsf{swap}^\oplus_{A,B} = \begin{bmatrix} 0 & \mathsf{id}_B \\ \mathsf{id}_A & 0  \end{bmatrix} \]
With this we can now define a symmetric monoidal structure on $\mathfrak{G}\left[ (\mathbb{X}, \dagger) \right]_X$. So define the monoidal product $\otimes$ on objects as $A \otimes B = A \oplus B$ and defined on maps $F=(f,p,s): A \to B$ and $G=(g,q,t): C \to D$ as follows: 
\begin{align*}
F \otimes G &=~ (f,p,x) \otimes (g,q,y) = (f \oplus g, p \oplus q, \begin{bmatrix} x \\ y
\end{bmatrix}) \\
&=~ \left( \begin{bmatrix} f & 0 \\ 
0 & g \end{bmatrix}, \begin{bmatrix} p & 0 \\ 
0 & q \end{bmatrix}, \begin{bmatrix} x \\ y
\end{bmatrix} \right): A \otimes B \to C \otimes D 
\end{align*}
The monoidal unit will be the zero object of $\mathbb{X}$, $I = \mathsf{0}$. The natural symmetry isomorphism $\mathsf{swap}_{A,B}: A \otimes B \to B \otimes A$ is defined as the triple:
\begin{align*}
\mathsf{swap}_{A,B} = \left(\mathsf{swap}^\oplus_{A,B}, 0, 0 \right) = \left( \begin{bmatrix} 0 & \mathsf{id}_B \\ \mathsf{id}_A & 0  \end{bmatrix}, \begin{bmatrix} 0 & 0 \\ 0 & 0  \end{bmatrix}, \begin{bmatrix} 0 \\ 0  \end{bmatrix} \right): A \otimes B \to B \otimes A 
\end{align*}

\begin{lemma} $\mathfrak{G}\left[ (\mathbb{X}, \dagger) \right]_X$ is a symmetric monoidal category. 
\end{lemma}
\begin{proof} We first need to explain why the monoidal product of maps is well-defined. This follows from the fact that the $\dagger$-biproduct of $\dagger$-positive maps is indeed $\dagger$-positive. To see this, consider $\dagger$-positive maps $p: A \to B$ and $q: C \to D$, with $p = \alpha^\dagger \circ \alpha$ for some map $\alpha: A \to W$ and $q = \beta^\dagger \circ \beta$ for some map $\beta: C \to V$. Then consider the map $\alpha \oplus \beta: A \oplus C \to W \oplus V$. We easily see that: 
\begin{align*}
(\alpha \oplus \beta)^\dagger \circ (\alpha \oplus \beta) &=~ \begin{bmatrix} \alpha & 0 \\ 
0 & \beta\end{bmatrix}^\dagger  \circ \begin{bmatrix} \alpha & 0 \\ 
0 &  \beta \end{bmatrix}  = \begin{bmatrix} \alpha^\dagger & 0 \\ 
0 & \beta^\dagger \end{bmatrix} \circ \begin{bmatrix} \alpha & 0 \\ 
0 &  \beta \end{bmatrix} \\
&=~ \begin{bmatrix} \alpha^\dagger \circ \alpha & 0 \\ 
0 & \beta^\dagger \circ \beta \end{bmatrix} = \begin{bmatrix} p & 0 \\ 
0 & q \end{bmatrix} = p \oplus q
\end{align*}
So $p \oplus q$ is $\dagger$-positive. As such, the monoidal product in $\mathfrak{G}\left[ (\mathbb{X}, \dagger) \right]_X$ is well-defined. Clearly, the natural symmetry isomorphism is also well-defined, so the symmetric monoidal structure on $\mathfrak{G}\left[ (\mathbb{X}, \dagger) \right]_X$ is well-defined. 

Next we need to prove that  $\otimes$ is functorial. Since $\mathsf{id}_{A \oplus B} = \mathsf{id}_A \oplus \mathsf{id}_B$, it is straightforward to see that $\mathsf{Id}_{A \otimes B} = \mathsf{Id}_A \otimes \mathsf{Id}_B$. To show that $\otimes$ also preserves composition, we will need the following identities that hold in any dagger additive category: 
\begin{align*}
(g \oplus k) \circ (f \oplus h) = (g \circ f) \oplus (k \circ h) &\quad& (f \oplus g)^\dagger  = f^\dagger \oplus g^\dagger\\
 (g \oplus f) + (k \oplus h) = (g + k) \oplus (f + h)  &\quad& (f \oplus g) \circ \begin{bmatrix} h \\ k
\end{bmatrix} = \begin{bmatrix} f \circ h \\ g \circ k 
\end{bmatrix}
\end{align*}
With these identities, we compute that: 
\begin{align*}
&(G \otimes K) \circ (F \otimes H) = \left( (g,q,y) \otimes (k, u, w) \right) \circ \left((f,p,x) \otimes (h, r, z) \right) \\
&= \left(g \oplus k, q \oplus u, \begin{bmatrix} y \\ w
\end{bmatrix} \right) \circ \left(f \oplus h, p \oplus r, \begin{bmatrix} x \\ z
\end{bmatrix} \right) \\
&= \fitline{\left( (g \oplus k) \circ (f \oplus h), (q \oplus u) + (g \oplus k) \circ (p \oplus r) \circ (g \oplus k)^\dagger, \begin{bmatrix} y \\ w
\end{bmatrix} + (g \oplus k) \circ \begin{bmatrix} x \\ z
\end{bmatrix} \right) }\\
&=  \fitline{\left( (g \oplus k) \circ (f \oplus h), (q \oplus u) +  (g \oplus k) \circ (p \oplus r) \circ (g^\dagger \oplus k^\dagger), \begin{bmatrix} y \\ w
\end{bmatrix} + \begin{bmatrix} g \circ x \\ k \circ z
\end{bmatrix} \right)}  \\
&=\left( (g \circ f) \oplus (k \circ h), (q \oplus u) + \left( (g \circ p \circ g^\dagger) \oplus (k \circ r \circ k^\dagger) \right), \begin{bmatrix} y + g \circ x \\ w + k \circ z
\end{bmatrix} \right) \\
&= \left( (g \circ f) \oplus (k \circ h),  (q + g \circ p \circ g^\dagger) \oplus (u + k \circ r \circ k^\dagger) , \begin{bmatrix} y + g \circ x \\ w + k \circ z
\end{bmatrix} \right) \\
&= (g \circ f, q + g \circ p \circ g^\dagger, y + g \circ x) \otimes(k \circ h, u + k \circ r \circ k^\dagger, w + k \circ z) \\ 
&=  \left( (g,q,y) \circ (f,p,x) \right) \otimes \left( (k, u, w) \circ (h, r, z) \right) = (G \circ F) \otimes (K \circ H)
\end{align*}
So $\otimes$ is indeed functorial. Next, since $\mathsf{swap}^\oplus_{B,A} \circ \mathsf{swap}^\oplus_{A,B} = \mathsf{id}_{A \oplus B}$, it  follows that $\mathsf{swap}_{B,A} \circ \mathsf{swap}_{A,B} = \mathsf{Id}_{A \otimes B}$.  It remains to show naturality. To prove this, we make use of the following identities that the symmetry map satisfies in any dagger additive category:
\begin{gather*}
(f \oplus g) \circ \mathsf{swap}_{A,B} = \mathsf{swap}_{D,C} \circ (g \oplus f) \qquad \mathsf{swap}^{\oplus~\dagger}_{A,B} = \mathsf{swap}^{\oplus}_{B,A} \\
   \mathsf{swap}^\oplus_{C,D} \circ \begin{bmatrix} y \\ x
\end{bmatrix} = \begin{bmatrix} x \\ y
\end{bmatrix}
\end{gather*}
Then we compute that: 
\begin{align*}
&\mathsf{swap}_{D,C} \circ (G \otimes F) = \left(\mathsf{swap}^\oplus_{C,D}, 0, 0 \right) \circ \left( (g,q,y) \otimes (f,p,x) \right)\\
 &= \left(\mathsf{swap}^\oplus_{C,D}, 0, 0 \right) \circ \left(g \oplus f, q \oplus p, \begin{bmatrix} y \\ x
\end{bmatrix} \right) \\
&= \left( \mathsf{swap}^\oplus_{D,C} \circ (g \oplus f), 0 + \mathsf{swap}^\oplus_{D,C} \circ (q \oplus p) \circ \mathsf{swap}^{\oplus~\dagger}_{D,C}, 0 + \mathsf{swap}^\oplus_{C,D} \circ \begin{bmatrix} y \\ x
\end{bmatrix} \right) \\ 
&= \left( (f \oplus g) \circ \mathsf{swap}^\oplus_{A,B}, \mathsf{swap}^\oplus_{D,C} \circ (q \oplus p) \circ \mathsf{swap}^\oplus_{C,D}, \mathsf{swap}^\oplus_{C,D} \circ \begin{bmatrix} y \\ x
\end{bmatrix} \right) \\
&= \left( (f \oplus g) \circ \mathsf{swap}^\oplus_{A,B}, (p \oplus q) \circ \mathsf{swap}^\oplus_{D,C} \circ \mathsf{swap}^\oplus_{C,D},  \begin{bmatrix} x \\ y
\end{bmatrix} \right) \\ 
&= \left( (f \oplus g) \circ \mathsf{swap}^\oplus_{A,B}, (p \oplus q),  \begin{bmatrix} x \\ y
\end{bmatrix} \right)\\
&= \left( (f \oplus g) \circ \mathsf{swap}^\oplus_{A,B}, (p \oplus q) + 0,  \begin{bmatrix} x \\ y
\end{bmatrix} + 0 \right) \\
&= \left( (f \oplus g) \circ \mathsf{swap}^\oplus_{A,B}, (p \oplus q) + (f \oplus g) \circ 0 \circ (f \oplus g)^\dagger, \begin{bmatrix} x \\ y
\end{bmatrix} + (f \oplus g) \circ 0 \right) \\
&= \left(f \oplus g, p \oplus q, \begin{bmatrix} s \\ t
\end{bmatrix}\right) \circ \left(\mathsf{swap}^\oplus_{A,B}, 0, 0 \right) = (F \otimes G) \circ \mathsf{swap}_{A,B} 
\end{align*}
Thus $\mathsf{swap}$ is indeed a natural isomorphism and its own inverse. Therefore, we conclude that $\mathfrak{G}\left[ (\mathbb{X}, \dagger) \right]_X$ is indeed a symmetric monoidal category.
\end{proof}

To show that $\mathfrak{G}\left[ (\mathbb{X}, \dagger) \right]_X$ is a Markov category, we also need to give the commutative comonoid structure. To do so, we use the fact that in an additive category, every object admits a canonical and unique cocommutative comonoid structure with respect to the biproduct. For every object $A \in \mathbb{X}$, let $\Delta_A: A \to A \oplus A$ be the canonical diagonal map of the biproduct, which is defined as follows: 
\[ \Delta_A = \begin{bmatrix} \mathsf{id}_A \\ \mathsf{id}_A
 \end{bmatrix} \]
 Then define the map $\mathsf{copy}_A: A \to A \otimes A$ in $\mathfrak{G}\left[ (\mathbb{X}, \dagger) \right]_X$ as the triple:
 \begin{align*}
\mathsf{copy}_A = (\Delta_A, 0, 0) = \left( \begin{bmatrix} \mathsf{id}_A \\ \mathsf{id}_A
 \end{bmatrix}, \begin{bmatrix} 0 & 0  \\ 0 & 0
 \end{bmatrix}, \begin{bmatrix} 0 \\ 0
 \end{bmatrix} \right) 
\end{align*}
Also define the map $\mathsf{del}_A: A \to \mathsf{0}$ in $\mathfrak{G}\left[ (\mathbb{X}, \dagger) \right]_X$ as the triple:
\begin{align*}
\mathsf{del}_A = (0,0,0) 
\end{align*}

\begin{theorem}\label{thm:gauss-markov} $\mathfrak{G}\left[ (\mathbb{X}, \dagger) \right]_X$ is a Markov category. 
\end{theorem}
\begin{proof} By construction, both $\mathsf{copy}$ and $\mathsf{del}$ are well-defined maps in $\mathfrak{G}\left[ (\mathbb{X}, \dagger) \right]_X$. Now recall that for every object $A \in \mathbb{X}$, $(A, \Delta_A, 0)$ is a cocommutative comonoid in $\mathbb{X}$. As such, it follows that $(A, \mathsf{copy}_A, \mathsf{del}_A)$ is a cocommutative comonoid in $\mathfrak{G}\left[ (\mathbb{X}, \dagger) \right]_X$. Moreover, since:
\[\Delta_{A \oplus B} = (\mathsf{id}_A \oplus \mathsf{swap}^\oplus_{A,B} \oplus \mathsf{id}_B) \circ (\Delta_A \otimes \Delta_B)\] 
it follows that $\mathsf{copy}$ and $\mathsf{del}$ are coherent with the monoidal structure in $\mathfrak{G}\left[ (\mathbb{X}, \dagger) \right]_X$ as well. Lastly, we clearly have that $\mathsf{del}_B \circ F = \mathsf{del}_A$, so every map in $\mathfrak{G}\left[ (\mathbb{X}, \dagger) \right]_X$ preserves the delete and thus $\mathsf{0}$ is a terminal object in $\mathfrak{G}\left[ (\mathbb{X}, \dagger) \right]_X$. So we conclude that $\mathfrak{G}\left[ (\mathbb{X}, \dagger) \right]_X$ is a indeed a Markov category.  
\end{proof}

It is important to note that while $\mathsf{0}$ is a zero object in $\mathbb{X}$, it is only a terminal object in $\mathfrak{G}\left[ (\mathbb{X}, \dagger) \right]_X$ and not necessarily an initial object. Indeed, for any $\dagger$-positive map $p: B \to B$ and any map $x: X \to B$, we have that $(0, p,x): \mathsf{0} \to B$ is a map in $\mathfrak{G}\left[ (\mathbb{X}, \dagger) \right]_X$. 

We now show that the deterministic maps in the Gauss construction are precisely those whose positive component is zero. 

\begin{lemma} A map $F = (f,p,x)$ in $\mathfrak{G}\left[ (\mathbb{X}, \dagger) \right]_X$ is deterministic if and only if $p=0$. 
\end{lemma}
\begin{proof} In an additive category, the following equalities hold: 
\begin{align*}
\Delta_B \circ f =  (f\oplus f) \circ \Delta_A &&  \begin{bmatrix} x \\ y \end{bmatrix} = (x \oplus y) \circ \Delta_X 
\end{align*}
So on the one hand we compute that: 
\begin{align*}
\mathsf{copy}_B \circ F &=~ (\Delta_B, 0, 0) \circ (f,p,x) \\
&=~ (\Delta_B \circ f, 0 + \Delta_B \circ p \circ \Delta^\dagger_B, 0 + \Delta_B \circ x) \\
&=~ ( (f\oplus f) \circ \Delta_A, (p \oplus p) \circ \Delta_B \circ \Delta^\dagger_B, (x \oplus x) \circ \Delta_X) \\
&=~ ( (f\oplus f) \circ \Delta_A, (p \oplus p) \circ \Delta_B \circ \Delta^\dagger_B, (x \oplus x) \circ \Delta_X) 
\end{align*}
While on the other hand we also compute that: 
\begin{align*}
&(F \otimes F) \circ \mathsf{copy}_A =~ \left( (f,p,x) \otimes (f,p,x) \right) \circ (\Delta_A, 0, 0) \\
&=~ \left(f \oplus f, p \oplus p, \begin{bmatrix} x \\ x\end{bmatrix} \right) \circ (\Delta_A, 0, 0)\\
&=~ \left( (f \oplus f) \circ \Delta_A, (p \oplus p) + (f \oplus f) \circ 0 \circ (f \oplus f)^\dagger,  \begin{bmatrix} x \\ x\end{bmatrix} + (f \oplus f) \circ 0 \right) \\
&=~ \left( (f \oplus f) \circ \Delta_A, p \oplus p,  (x \oplus x) \circ \Delta_X \right) 
\end{align*}
So we have the following two equalities: 
\begin{align*}
(F \otimes F) \circ \mathsf{copy}_A &= \left( (f \oplus f) \circ \Delta_A, (p \oplus p),  (x \oplus x) \circ \Delta_X  \right) \\
\mathsf{copy}_B \circ F  &=  ( (f\oplus f) \circ \Delta_A, (p \oplus p) \circ \Delta_B \circ \Delta^\dagger_B, (x \oplus x) \circ \Delta_X)
\end{align*}
So if $p=0$ we clearly have that $(F \otimes F) \circ \mathsf{copy}_A = \mathsf{copy}_B \circ F$. Thus $F = (f, 0, x)$ is deterministic. Conversely, if $F = (f,p,x)$ is deterministic, then $(F \otimes F) \circ \mathsf{copy}_A = \mathsf{copy}_B \circ F$, which in particular implies that $p \oplus p = (p \oplus p) \circ \Delta_B \circ \Delta^\dagger_B$. Now in matrix representation, this gives us the following:
\begin{align*}
\begin{bmatrix} p & 0 \\ 
0 & p
\end{bmatrix} 
&= p \oplus p = (p \oplus p) \circ \Delta_B \circ \Delta^\dagger_B = \begin{bmatrix} p & 0 \\ 
0 & p
\end{bmatrix} \circ \begin{bmatrix} \mathsf{id}_B  \\ 
 \mathsf{id}_B 
\end{bmatrix} \circ \begin{bmatrix} \mathsf{id}_B  \\ 
 \mathsf{id}_B 
\end{bmatrix}^\dagger  \\
&= \begin{bmatrix} p & 0 \\ 
0 & p
\end{bmatrix} \circ \begin{bmatrix} \mathsf{id}_B  \\ 
 \mathsf{id}_B 
\end{bmatrix} \circ \begin{bmatrix} \mathsf{id}^\dagger_B  &  
 \mathsf{id}^\dagger_B 
\end{bmatrix} =  \begin{bmatrix} p & 0 \\ 
0 & p
\end{bmatrix} \circ \begin{bmatrix} \mathsf{id}_B  \\ 
 \mathsf{id}_B 
\end{bmatrix} \circ \begin{bmatrix} \mathsf{id}_B  &  
 \mathsf{id}_B 
\end{bmatrix} \\
&= \begin{bmatrix} p & 0 \\ 
0 & p
\end{bmatrix} \circ \begin{bmatrix} \mathsf{id}_B & \mathsf{id}_B \\ 
\mathsf{id}_B & \mathsf{id}_B
\end{bmatrix} = \begin{bmatrix} p & p \\ 
p & p
\end{bmatrix} 
\end{align*}
So $\begin{bmatrix} p & 0 \\ 
0 & p
\end{bmatrix} = \begin{bmatrix} p & p \\ 
p & p
\end{bmatrix} $, which then implies that $p=0$. 
\end{proof}

Applying the Gauss construction to the category of real matrices recaptures the example of real Gaussian  probability theory. By an isomorphism of Markov categories, we mean an isomorphism in the 2-category of Markov categories in the sense of Fritz \cite[Sec 10]{fritz2020synthetic}, which simply means an isomorphism of categories which preserves the monoidal structure and comonoid structure on the nose. 

\begin{example}\label{ex:Gauss-real-mat} Consider $(\mathsf{MAT}(\mathbb{R}), \mathsf{T})$ and take $X =1$. Then we have an isomorphism of Markov categories $\mathfrak{G}\left[ (\mathsf{MAT}(\mathbb{R}), \mathsf{T}) \right]_1 \cong \mathsf{Gauss}$. 
\end{example}

In the above example, we chose  $X=1$  as the base object in Gauss construction; however, it is natural to ask what happens when we take our base object to be some other $n \in \mathbb{N}$. There are really two cases, when $n=0$ and when $n \geq 1$. Indeed, note that $n=0$ is the zero object; on the other hand, when $n \geq 1$, it is the biproduct of $n$ copies of $1$, $n = 1 \oplus  \hdots \oplus 1$. So we will now explain what we get when we take the Gauss construction with base object the zero object or a biproduct of objects. 

First we consider the case of the zero object $\mathsf{0}$. Since $\mathsf{0}$ is an initial object, there is only map out of $\mathsf{0}$, namely the zero map $0: \mathsf{0} \to B$. Therefore, every map in $\mathfrak{G}\left[ (\mathbb{X}, \dagger) \right]_{\mathsf{0}}$ is a triple of the form $F = (f, p, 0)$. As such, in this case, we can consider our maps to simply be pairs instead of triples. So define $\mathfrak{P}\left[ (\mathbb{X}, \dagger) \right]$ as the category whose objects are the same as the objects of $\mathbb{X}$ and where a map $F: A \to B$ in $\mathfrak{P}\left[ (\mathbb{X}, \dagger) \right]$ is a pair $F=(f, p): A \to B$ consisting of a map $f: A \to B$ and a $\dagger$-positive map $p: B \to B$. Composition, identities, the monoidal structure, and the comonoid structure are defined as in the Gauss construction but ignoring the third component of maps. 

\begin{lemma} $\mathfrak{P}\left[ (\mathbb{X}, \dagger) \right]$ is a Markov category and moreover, we have an isomorphism of Markov categories $\mathfrak{P}\left[ (\mathbb{X}, \dagger) \right] \cong \mathfrak{G}\left[ (\mathbb{X}, \dagger) \right]_{\mathsf{0}}$. 
\end{lemma}

\begin{example}\label{ex:Gauss-real-mat-0} Recall that the zero object in $\mathsf{MAT}(\mathbb{R})$ is $n=0$. Then $\mathfrak{P}\left[ (\mathsf{MAT}(\mathbb{R}), \mathsf{T}) \right] \cong \mathfrak{G}\left[ (\mathsf{MAT}(\mathbb{R}), \mathsf{T}) \right]_0$ captures \textit{centered} Gaussian conditional distributions where the expected value of the Gaussian noise is fixed to be \(0\), so in other words, \textit{centred} Gaussian conditional distributions. Indeed, define $\mathsf{Gauss}_0$ to be the category whose objects are the natural numbers $n \in \mathbb{N}$ and where a map is a pair $(M,C): n \to m$ consisting of an $m \times n$ $\mathbb{R}$-matrix $M$ and a positive semidefinite square $m \times m$ $\mathbb{R}$-matrix $C$. We interpret this pair as a centered Gaussian conditional distribution of a random variable $Y = MX + \xi$, where $\xi \in \mathbb{R}^m$ is Gaussian noise independent from $X$, with expected value $\mathsf{E}[\xi] = 0$. Then $\mathfrak{P}\left[ (\mathsf{MAT}(\mathbb{R}), \mathsf{T}) \right] \cong \mathsf{Gauss}_0$. 
\end{example}

Now consider what happens when we take the base object to be a biproduct $X_1 \oplus \hdots \oplus X_n$. By the couniversal property of the coproduct, a map $X_1 \oplus \hdots \oplus X_N \to B$ is fully determined by $n$ maps $X_j \to B$. Therefore, we can instead interpret maps in $\mathfrak{G}\left[ (\mathbb{X}, \dagger) \right]_{X_1 \oplus \hdots \oplus X_n}$ as $2+n$-tuples. So define $\mathfrak{G}\left[ (\mathbb{X}, \dagger) \right]_{X_1, \hdots, X_n}$ to be the category whose objects are the same as the objects of $\mathbb{X}$ and where a map $F: A \to B$ in $\mathfrak{G}\left[ (\mathbb{X}, \dagger) \right]_{X_1, \hdots, X_n}$ is a $2+n$-tuple $F=(f, p, x_1, \hdots, x_n): A \to B$ consisting of a map $f: A \to B$, a $\dagger$-positive map $p: B \to B$, and $n$ maps $x_j: X_j \to B$ $(1 \leq j \leq n)$. Composition, identities, the monoidal structure, and the comonoid structure are defined as in the Gauss construction in the first two arguments of maps and as in the third argument of maps for the remaining $n$ arguments. 

\begin{lemma} $\mathfrak{G}\left[ (\mathbb{X}, \dagger) \right]_{X_1, \hdots, X_n}$ is a Markov category and we have an isomorphism of Markov categories $\mathfrak{G}\left[ (\mathbb{X}, \dagger) \right]_{X_1, \hdots, X_n} \cong \mathfrak{G}\left[ (\mathbb{X}, \dagger) \right]_{X_1 \oplus \hdots \oplus X_n}$. 
\end{lemma}

\begin{example}\label{ex:Gauss-real-mat-1}  For $k \geq 1$, which is a biproduct $k = 1 \oplus \hdots \oplus 1$ in $\mathsf{MAT}(\mathbb{R})$, we have that $\mathfrak{G}\left[ (\mathsf{MAT}(\mathbb{R}), \mathsf{T}) \right]_k \cong \mathfrak{G}\left[ (\mathsf{MAT}(\mathbb{R}), \mathsf{T}) \right]_{1, \hdots, 1}$ which we can interpret as $k$ Gaussian conditional distributions which differ only by their Gaussian noises who have the same covariance, but possibly different expected values. Indeed, define $\mathsf{Gauss}_k$ to be the category whose objects are the natural numbers $n \in \mathbb{N}$ and where a map is an $k$-tuple $(M,C, s_1, \hdots, s_k): n \to m$ consisting of an $m \times n$ $\mathbb{R}$-matrix $M$, a positive semidefinite square $m \times m$ $\mathbb{R}$-matrix $C$, and $k$ (column) vector $s_1, \hdots, s_k \in \mathbb{R}^m$. We interpret this tuple as Gaussian distributions of random variables $Y_1 = MX_1 + \xi_1$, ..., $Y_k = MX_k + \xi_k$, where each $\xi_j \in \mathbb{R}^m$ is Gaussian noise independent from $X_j$ with $\mathsf{E}[\xi_j] = s_j$ but which all have the same covariance matrix $\mathsf{VAR}[\xi_1] = \hdots = \mathsf{VAR}[\xi_k] = C$. Then we have that $\mathfrak{G}\left[ (\mathsf{MAT}(\mathbb{R}), \mathsf{T}) \right]_k \cong \mathsf{Gauss}_k$. 
\end{example}

 Of course, the main purpose for introducing the Gauss construction is to develop new examples of Markov categories. So let us now apply the Gauss construction to complex matrices and explain how this corresponds to \emph{proper} complex Gaussian conditional distributions. 
 
\begin{example}\label{ex:GaussC}
Define \(\mathsf{Gauss}_p^{\mathbb C} \coloneqq \mathfrak{G}\left[ (\mathsf{MAT}(\mathbb{C}), \ast) \right]_1\).
 By  Thm~\ref{thm:gauss-markov}, \(\mathsf{Gauss}_p^{\mathbb C}\) is a Markov category.  Just as the maps in $\mathsf{Gauss}$ were interpreted as real Gaussian distributions, we claim that maps in $\mathsf{Gauss}_p^\mathbb{C}$ correspond to certain \emph{complex Gaussian distributions}. Indeed, recall that a complex Gaussian random variable $\zeta$ is determined by its expectation $\mathsf{E}[\zeta]$, its covariance and \emph{pseudocovariance}
 \[
 \mathsf{VAR}(\zeta)
\coloneqq \mathsf{E}[(\zeta - E[\zeta])(\zeta - E[\zeta])^{\ast}],
\quad 
\mathsf{PVAR}(\zeta) 
\coloneqq \mathsf{E}[(\zeta - E[\zeta])(\zeta - E[\zeta])^{\mathsf T}].
\]
 A complex Gaussian distribution $\zeta$ is \textbf{proper}~\cite[Def 2.1]{Schreier2010} when its pseudocovariance is zero, $\mathsf{PVAR}(\zeta)=0$. Therefore, arbitrary morphisms $(M,C,s): n \to m$ in $\mathsf{Gauss}_p^\mathbb{C}$ are interpreted as proper complex Gaussian conditional distributions $Y = M X + \zeta$ where $\zeta \in \mathbb{C}^m$ is a complex Gaussian random variable with $\mathsf{E}[\zeta]=s$, $\mathsf{VAR}(\zeta)=C$, and $\mathsf{PVAR}(\zeta)=0$. When moreover, the variance is zero, then such a proper complex Gaussian distribution is said to be \textbf{circularly symmetric}~\cite[Result 2.11]{Schreier2010}, so that it is invariant under the action of the complex unit circle. As such the morphisms in \(\mathfrak{P}\left[ (\mathsf{MAT}(\mathbb{C}), \ast) \right]\) are precisely the circularly symmetric complex Gaussian conditional distributions.
\end{example}

We now also apply the Gauss construction to quaternionic matrices, which will result in certain \emph{quaternionic Gaussian distributions}. 

\begin{example}\label{ex:GaussH} Define \( \mathsf{Gauss}_p^{\mathbb H}\coloneqq \mathfrak{G}\left[ (\mathsf{MAT}(\mathbb{H}), \ast) \right]_1\).
Just as the maps in $\mathsf{Gauss}_p^{\mathbb{C}}$ were interpreted as proper complex Gaussian distributions, we claim that maps in $\mathsf{Gauss}_p^\mathbb{H} $ correspond to the quaternionic analogue. This time, a quaternionic Gaussian distribution is determined  by its expectation, its covariance, and its three \emph{complementary covariances} \cite[Sec~III.A.]{Via2010}, which are quaternionic analogues of the complex pseudocovariance. So given a quaternion random variable \(\zeta\in \mathbb{H}^m\), with expectation \(\mathsf{E}[\zeta]\), its covariance and complementary covariances, for all \(\eta\in\{\mathbf{i},\mathbf{j},\mathbf{k}\}\), are given by:
\[
\mathsf{VAR}(\zeta)
\coloneqq \mathsf{E}\bigl[(\zeta-\mathsf{E}[\zeta])(\zeta-\mathsf{E}[\zeta])^{*}\bigr],
\;
\mathsf{PVAR}_\eta(\zeta) \coloneqq \mathsf{E}\bigl[(\zeta-\mathsf{E}[\zeta])\bigl(\bigl(\zeta-\mathsf{E}[\zeta]\bigr)^{(\eta)}\bigr)^{*}\bigr],
\]
where given a quaternionic matrix \(A\),  \(A^{(\eta)}\) is the quaternionic matrix given by  linearly extending the conjugation \((-)^{(\eta)}: \mathbb{H}\to \mathbb{H}\) given by \(q \mapsto -\eta q \eta\). A quaternionic Gaussian distribution is called \textbf{proper}\footnote{We note that in almost all of the signal processing literature, where quaternionic Gaussian distributions seem to be most commonly used, the expectation of proper quaternionic Gaussian distributions is assumed to be zero for convenience. However, as mentioned in~\cite[Footnote~2]{Via2010}, the definition of proper quaternionic Gaussian distribution with nontrivial expectation which we have given above follows easily.} when all three complementary covariances vanish \cite[Sec~III.D]{Via2010}, so that  \(\mathsf{PVAR}_{\mathbf{i}}(\zeta)=\mathsf{PVAR}_{\mathbf{j}}(\zeta)=\mathsf{PVAR}_{\mathbf{k}}(\zeta)=0\). Thus a general morphism \((M,C,s):n\to m\) in \(\mathsf{Gauss}_p^{\mathbb H}\) can be interpreted as a proper quaternionic Gaussian conditional distribution \(Y = MX + \zeta\), where \(\zeta\in\mathbb{H}^m\) is a  quaternionic Gaussian random variable with \(\mathsf{E}[\zeta]=s\), \(\mathsf{VAR}(\zeta)=C\), and \(\mathsf{PVAR}_\eta(\zeta)=0\) for all \(\eta\in\{\mathbf{i},\mathbf{j},\mathbf{k}\}\). When moreover, the variance is zero, then such a proper quaternion  Gaussian distribution is invariant under a canonical action of the quaternion unit sphere \cite[Lemma 9]{Via2010}. As such the morphisms in \(\mathfrak{P}\left[ (\mathsf{MAT}(\mathbb{Q}), \ast) \right]\) are precisely the \emph{spherically symmetric} quaternionic Gaussian conditional distributions. 
\end{example}

 More generally, by applying the Gauss construction to matrices over an (involutive) ring $R$ results in a Markov category one can then investigate whether this recaptures some already established theory of Gaussian distributions over $R$ or perhaps instead opens the door to consider such. Of course, not every dagger additive category is a category of matrices over a ring, and so it would also be interesting to study what frameworks for probability one obtains when applying the Gauss construction to non-matrix category examples, such as the category of Hilbert spaces. 

We conclude this section with a brief discussion about the functoriality of the Gauss construction. For dagger additive categories $(\mathbb{X}, \dagger)$ and $(\mathbb{Y}, \dagger)$, a \textit{dagger additive functor} between them is functor $\mathcal{F}: \mathbb{X} \to \mathbb{Y}$ which is a \textbf{dagger functor} \cite[Def 2.3]{heunen2016monads}, that is, it preserves the dagger $\mathcal{F}(-^\dagger) = \mathcal{F}(-)^\dagger$, and is an \textbf{additive functor} \cite[Prop 1.3.4]{borceux1994handbook}, which means it preserves finite biproducts (up to canonical isomorphism) \cite[Def 2.22]{heunen2019categories} or equivalently that it preserves the group structure of the homsets \cite[Def 1.3.1]{borceux1994handbook}. On the other hand, for Markov categories $\mathbb{C}$ and $\mathbb{D}$, a \textit{Markov functor} \cite[10.14]{fritz2020synthetic} is a \textbf{strong symmetric monoidal functor} \cite[Def 1.29]{heunen2019categories} which preserves the comonoid structure (up to the monoidal isomorphisms of the monoidal functor). Now for a dagger additive functor $\mathcal{F}: (\mathbb{X}, \dagger) \to (\mathbb{Y}, \dagger)$, for any object $X \in \mathbb{X}$, define the functor $\mathfrak{G}[\mathcal{F}]_X: \mathfrak{G}\left[ (\mathbb{X}, \dagger) \right]_{X} \to \mathfrak{G}\left[ (\mathbb{Y}, \dagger) \right]_{\mathcal{F}(X)}$ on objects as $\mathfrak{G}[\mathcal{F}]_X(A) = \mathcal{F}(A)$ and on maps as $\mathfrak{G}[\mathcal{F}]_X(f,p,x) = \left( \mathcal{F}(f), \mathcal{F}(p), \mathcal{F}(x) \right)$. 

\begin{lemma} If $\mathcal{F}: (\mathbb{X}, \dagger) \to (\mathbb{Y}, \dagger)$ is a dagger additive functor, then for any object $X \in \mathbb{X}$, $\mathfrak{G}[\mathcal{F}]_X: \mathfrak{G}\left[ (\mathbb{X}, \dagger) \right]_{X} \to \mathfrak{G}\left[ (\mathbb{X}, \dagger) \right]_{\mathcal{F}(X)}$ is a Markov functor. 
\end{lemma}

It is straightforward to see that the Gauss construction then gives a \mbox{(2-)}functor from the (2-)category of dagger additive categories with a chosen object to the (2-)category of Markov categories. Now if we wish to do away from having to choose an object, since dagger additive functors preserve zero objects (up to isomorphism) and zero maps, it follows that if $\mathcal{F}: (\mathbb{X}, \dagger) \to (\mathbb{Y}, \dagger)$ is a dagger additive functor, then $\mathfrak{P}[\mathcal{F}]: \mathfrak{P}\left[ (\mathbb{X}, \dagger) \right] \to \mathfrak{P}\left[ (\mathbb{Y}, \dagger) \right]$ is a Markov functor, which is defined on objects as $\mathfrak{P}[\mathcal{F}](A) = \mathcal{F}(A)$ and on maps as $\mathfrak{P}[\mathcal{F}](f,p) = \left( \mathcal{F}(f), \mathcal{F}(p) \right)$. Therefore, the $\mathfrak{P}$ construction gives a (2-)functor from the actual (2-)category of dagger additive categories to the (2-)category of Markov categories.

\section{Conditionals in the Gauss Construction}\label{sec:cond-in-gauss}

In this section, we characterize all possible conditionals in the Gauss construction and show that they are fully determined by a map in the base category which is compatible with the positive component of maps in the Gauss construction. We call such a map a \textit{conditional generator}, and we will show that conditionals in the Gauss construction are in bijective correspondence with conditional generators in the base category. 

Throughout this section, let $(\mathbb{X}, \dagger)$ be a dagger additive category and fix an object $X \in \mathbb{X}$. Consider first a map in $\mathfrak{G}\left[ (\mathbb{X}, \dagger) \right]_X$ for which we may take its conditional, that is, a map of type $F: A \to B \otimes C$ in $\mathfrak{G}\left[ (\mathbb{X}, \dagger) \right]_X$. Then $F$ is a triple consisting of a map of type $A \to B \oplus C$, a $\dagger$-positive map $B \oplus C \to B\oplus C$, and a generalized point $X \to B \oplus C$. Using matrix representation, $F$ is thus a triple consisting of a $2 \times 1$ matrix, $2 \times 2$ matrix, and a $2 \times 1$ matrix whose coefficients are maps of the following type: 
 \begin{gather*}
\begin{array}[c]{c}F= \left( \begin{bmatrix} f: A \to B \\ g: A \to C
 \end{bmatrix}, \begin{bmatrix} \alpha: B \to B & \beta: C \to B \\
\gamma: B \to C & \delta: C \to C
\end{bmatrix}, \begin{bmatrix} s: X \to B \\ t: X \to C
 \end{bmatrix}  \right)
\end{array} 
 \end{gather*} 
 Moreover, since the second component   is $\dagger$-positive, there exists a map type $\begin{bmatrix} \phi & \psi \end{bmatrix}: B \oplus C \to D$, together with maps $\phi: B \to D$ and $\psi: C \to D$, such that: 
 \[ \begin{bmatrix} \alpha & \beta \\
\gamma & \delta 
\end{bmatrix} = \begin{bmatrix} \phi & \psi \end{bmatrix}^\dagger \circ \begin{bmatrix} \phi & \psi \end{bmatrix} = \begin{bmatrix} \phi^\dagger \\ \psi^\dagger \end{bmatrix} \circ \begin{bmatrix} \phi & \psi \end{bmatrix} = \begin{bmatrix} \phi^\dagger \circ \phi & \phi^\dagger \circ \psi \\
\psi^\dagger \circ \phi & \psi^\dagger \circ \psi 
\end{bmatrix}  \]
This implies that $\alpha$ and $\delta$ are both $\dagger$-positive maps, and also that $\gamma = \beta^\dagger$. So our map $F: A \to B \otimes C$ takes the form: 
\[F= \left( \begin{bmatrix} f \\ g
 \end{bmatrix}, \begin{bmatrix} \alpha & \beta \\
\beta^\dagger & \delta 
\end{bmatrix}, \begin{bmatrix} s \\ t
 \end{bmatrix}  \right) \]
On the other hand, a potential conditional of $F$ would be a map of type $G: B \otimes A \to C$ in $\mathfrak{G}\left[ (\mathbb{X}, \dagger) \right]_X$, which is a triple  consisting of a map of type $B \oplus A \to C$, a $\dagger$-positive map $C \to C$, and a generalized point $X \to C$. In matrix representation, $G$ is a triple consisting of a $1 \times 2$ matrix and two $1 \times 1$ matrices who coefficients are:
 \begin{gather*}
\begin{array}[c]{c} G = \left( \begin{bmatrix} m: B \to C & k: A \to C
 \end{bmatrix}, \eta: C \to C, u: X \to C  \right)
\end{array} 
 \end{gather*} 
where $\eta$ is $\dagger$-positive. We now show that if $G$ is a conditional of $F$, then every map in $G$ can be expressed using the data of $F$ and $m$.

\begin{lemma}\label{lem:con-m1} Suppose that $G = \left( \begin{bmatrix} m & k
 \end{bmatrix}, \eta, u  \right)$ is a conditional of $F= \left( \begin{bmatrix} f \\ g
 \end{bmatrix}, \begin{bmatrix} \alpha & \beta \\
\beta^\dagger & \delta 
\end{bmatrix}, \begin{bmatrix} s \\ t
 \end{bmatrix}  \right)$. Then the following equalities hold: 
   \begin{align*}
 k = g -  m \circ f  &\quad & m \circ \alpha = \beta^\dagger &\quad & \eta  = \delta - m \circ \beta  &\quad & u = t - m\circ s 
 \end{align*}
\end{lemma}
\begin{proof} Suppose that $G$ is a conditional of $F$, which means that the following equality holds: 
\[ (\mathsf{Id}_B \otimes G) \circ (\mathsf{copy}_B \otimes \mathsf{Id}_A) \circ (\mathsf{Id}_B \otimes \mathsf{del}_C \otimes \mathsf{Id}_A) \circ (F \otimes \mathsf{Id}_A) \circ \mathsf{copy}_A = F \]
Let us first write down each component of the left-hand side as a triple: 
\begin{align*}
\mathsf{Id}_B \otimes G &= (\mathsf{id}_B, 0, 0) \otimes \left( \begin{bmatrix} m & k
 \end{bmatrix}, \eta, u  \right) \\
 &= \left( \begin{bmatrix} \mathsf{id}_B & 0 & 0 \\ 0 & m & k \end{bmatrix}, \begin{bmatrix} 0 & 0 \\ 0 & \eta \end{bmatrix}, \begin{bmatrix} 0 \\ u \end{bmatrix} \right),\\[1ex]
\mathsf{copy}_B \otimes \mathsf{Id}_A &= \left( \begin{bmatrix} \mathsf{id}_B \\ \mathsf{id}_B
 \end{bmatrix}, \begin{bmatrix} 0 & 0  \\ 0 & 0
 \end{bmatrix}, \begin{bmatrix} 0 \\ 0
 \end{bmatrix} \right)  \otimes (\mathsf{id}_A, 0, 0) \\
 &=  \left( \begin{bmatrix} \mathsf{id}_B & 0 \\ \mathsf{id}_B & 0 \\ 0 & \mathsf{id}_A \end{bmatrix}, \begin{bmatrix} 0 & 0 & 0\\ 0 & 0 & 0 \\ 0 & 0 & 0 \end{bmatrix}, \begin{bmatrix} 0\\ 0 \\ 0  \end{bmatrix} \right), \\[1ex]
  \mathsf{Id}_B \otimes \mathsf{del}_C \otimes \mathsf{Id}_A &=  (\mathsf{id}_B, 0, 0) \otimes (0,0,0)  \otimes  (\mathsf{id}_A, 0, 0) \\
&= \left( \begin{bmatrix} \mathsf{id}_B & 0 & 0\\ 0 & 0 & \mathsf{id}_A \end{bmatrix}, \begin{bmatrix} 0 & 0 \\ 0 & 0  \end{bmatrix}, \begin{bmatrix} 0\\ 0   \end{bmatrix} \right), \\[1ex]
F \otimes \mathsf{Id}_A &= \left( \begin{bmatrix} f \\ g
 \end{bmatrix}, \begin{bmatrix} \alpha & \beta \\
\beta^\dagger & \delta 
\end{bmatrix}, \begin{bmatrix} s \\ t
 \end{bmatrix}  \right) \otimes (\mathsf{id}_A, 0, 0) \\
 &= \left( \begin{bmatrix} f & 0 \\ g & 0 \\ 0 & \mathsf{id}_A
 \end{bmatrix}, \begin{bmatrix} \alpha & \beta & 0  \\
\beta^\dagger & \delta & 0 \\
0 & 0 & 0
\end{bmatrix}, \begin{bmatrix} s \\ t \\ 0 
 \end{bmatrix}  \right),\\[1ex]
 \mathsf{copy}_A &= \left( \begin{bmatrix} \mathsf{id}_A \\ \mathsf{id}_A
 \end{bmatrix}, \begin{bmatrix} 0 & 0  \\ 0 & 0
 \end{bmatrix}, \begin{bmatrix} 0 \\ 0
 \end{bmatrix} \right).
 \end{align*}
Let us compose them together part by part, we leave checking the details as an exercise for the reader. So we can first compute that: 
 \begin{align*}
(F \otimes \mathsf{Id}_A) \circ \mathsf{copy}_A &=~ \left( \begin{bmatrix} f & 0 \\ g & 0 \\ 0 & \mathsf{id}_A
 \end{bmatrix}, \begin{bmatrix} \alpha & \beta & 0  \\
\beta^\dagger & \delta & 0 \\
0 & 0 & 0
\end{bmatrix}, \begin{bmatrix} s \\ t \\ 0 
 \end{bmatrix}  \right) \circ \left( \begin{bmatrix} \mathsf{id}_A \\ \mathsf{id}_A
 \end{bmatrix}, \begin{bmatrix} 0 & 0  \\ 0 & 0
 \end{bmatrix}, \begin{bmatrix} 0 \\ 0
 \end{bmatrix} \right)  \\
 &=~  \left( \begin{bmatrix} f \\ g \\ \mathsf{id}_A 
 \end{bmatrix}, \begin{bmatrix} \alpha & \beta & 0  \\
\beta^\dagger & \delta & 0 \\
0 & 0 & 0
\end{bmatrix}, \begin{bmatrix} s \\ t \\ 0 
 \end{bmatrix}  \right)
\end{align*} 
Then we get: 
\begin{align*}
& (\mathsf{Id}_B \otimes \mathsf{del}_C \otimes \mathsf{Id}_A) \circ \left( \begin{bmatrix} f \\ g \\ \mathsf{id}_A 
 \end{bmatrix}, \begin{bmatrix} \alpha & \beta & 0  \\
\beta^\dagger & \delta & 0 \\
0 & 0 & 0
\end{bmatrix}, \begin{bmatrix} s \\ t \\ 0 
 \end{bmatrix}  \right)\\
 &=~ \left( \begin{bmatrix} f \\ \mathsf{id}_A
 \end{bmatrix}, \begin{bmatrix} \alpha & \beta & 0 \\ 0 & 0 & 0 
\end{bmatrix} \circ \begin{bmatrix} \mathsf{id}_B & 0 \\ 0 & 0 \\ 0 & \mathsf{id}_A \end{bmatrix},\begin{bmatrix} s \\ 0 
 \end{bmatrix}  \right) =~ \left( \begin{bmatrix} f \\ \mathsf{id}_A
 \end{bmatrix}, \begin{bmatrix} \alpha & 0 \\ 0 & 0  \end{bmatrix},\begin{bmatrix} s \\ 0 
 \end{bmatrix}  \right) 
\end{align*}
Then we compute:
\begin{align*}
&(\mathsf{copy}_B \otimes \mathsf{Id}_A) \circ \left( \begin{bmatrix} f \\ \mathsf{id}_A
 \end{bmatrix}, \begin{bmatrix} \alpha & 0 \\ 0 & 0  \end{bmatrix},\begin{bmatrix} s \\ 0 
 \end{bmatrix}  \right) \\
 &=~ \left( \begin{bmatrix} \mathsf{id}_B & 0 \\ \mathsf{id}_B & 0 \\ 0 & \mathsf{id}_A \end{bmatrix}, \begin{bmatrix} 0 & 0 & 0\\ 0 & 0 & 0 \\ 0 & 0 & 0 \end{bmatrix}, \begin{bmatrix} 0\\ 0 \\ 0  \end{bmatrix} \right)  \circ \left( \begin{bmatrix} f \\ \mathsf{id}_A
 \end{bmatrix}, \begin{bmatrix} \alpha & 0 \\ 0 & 0  \end{bmatrix},\begin{bmatrix} s \\ 0 
 \end{bmatrix}  \right) \\
&=~ \left( \begin{bmatrix}f \\ f \\ \mathsf{id}_A 
 \end{bmatrix}, \begin{bmatrix} \alpha & \alpha & 0 \\ \alpha & \alpha & 0 \\ 0 & 0 & 0  \end{bmatrix},  \begin{bmatrix} s \\  s \\ 0 
 \end{bmatrix} \right) 
\end{align*}
Lastly we finally get that:  
 \begin{align*}
& (\mathsf{Id}_Z \otimes G) \circ \left( \begin{bmatrix}f \\ f \\ \mathsf{id}_A 
 \end{bmatrix}, \begin{bmatrix} \alpha & \alpha & 0 \\ \alpha & \alpha & 0 \\ 0 & 0 & 0  \end{bmatrix},  \begin{bmatrix} s \\  s \\ 0 
 \end{bmatrix} \right) \\
 &=~ \left( \begin{bmatrix} \mathsf{id}_B & 0 & 0 \\ 0 & m & k \end{bmatrix}, \begin{bmatrix} 0 & 0 \\ 0 & \eta \end{bmatrix}, \begin{bmatrix} 0 \\ u \end{bmatrix} \right) \circ \left( \begin{bmatrix}f \\ f \\ \mathsf{id}_A 
 \end{bmatrix}, \begin{bmatrix} \alpha & \alpha & 0 \\ \alpha & \alpha & 0 \\ 0 & 0 & 0  \end{bmatrix},  \begin{bmatrix} s \\  s \\ 0 
 \end{bmatrix} \right) \\
   &=~ \left( \begin{bmatrix} f \\ m \circ f + k 
 \end{bmatrix}, \begin{bmatrix} \alpha & \alpha \circ m^\dagger \\ m \circ \alpha &  \eta + m \circ \alpha\circ m^\dagger \end{bmatrix},  \begin{bmatrix} s \\ u + m \circ s \end{bmatrix} \right)
\end{align*}  
Thus we end up with: 
\begin{equation}\begin{gathered}\label{longcalc}
(\mathsf{Id}_B \otimes G) \circ (\mathsf{copy}_B \otimes \mathsf{Id}_A) \circ (\mathsf{Id}_B \otimes \mathsf{del}_C \otimes \mathsf{Id}_A) \circ (F \otimes \mathsf{Id}_A) \circ \mathsf{copy}_A\\
   = \left( \begin{bmatrix} f \\ m \circ f + k 
 \end{bmatrix}, \begin{bmatrix} \alpha & \alpha \circ m^\dagger \\ m \circ \alpha &  \eta + m \circ \alpha\circ m^\dagger \end{bmatrix},  \begin{bmatrix} s \\ u + m \circ s \end{bmatrix} \right)
\end{gathered}\end{equation}  
Then since $G$ is a conditional of $F$,  the following equality holds: 
\[\fitline{
\left( \begin{bmatrix} f \\ m \circ f + k 
 \end{bmatrix}, \begin{bmatrix} \alpha & \alpha \circ m^\dagger \\ m \circ \alpha &  \eta + m \circ \alpha\circ m^\dagger \end{bmatrix},  \begin{bmatrix} s \\ u + m \circ s \end{bmatrix} \right)
 = \left( \begin{bmatrix} f \\ g
 \end{bmatrix}, \begin{bmatrix} \alpha & \beta \\
\beta^\dagger & \delta 
\end{bmatrix}, \begin{bmatrix} s \\ t
 \end{bmatrix}  \right)}\]
 In particular this implies that $m \circ f + k = g$, $m\circ \alpha = \beta^\dagger$, $\eta + m \circ \alpha\circ m^\dagger  = \delta$, and $u + m\circ s = t$. Now since $\alpha$ is $\dagger$-positive, we have that $\alpha^\dagger = \alpha$, and so $m\circ \alpha = \beta^\dagger$ implies that $\alpha \circ m^\dagger = \beta$ and thus we also get that $\eta + m \circ \beta = \delta$. Then rearranging by subtracting where appropriate, we get $k = g - m\circ f$, $m\circ \alpha = \beta^\dagger$, $\eta   = \delta - m \circ \alpha\circ m^\dagger$, and $u = t - m\circ s$ as desired. 
\end{proof}

The previous lemma tells us that if $G: B \otimes A \to C$ is a conditional of $F: A \to B \otimes C$, then every component of $F$ and the $B \to C$ component of $G$ fully determine the other components of $G$. Therefore, we may ask what conditions on a map $B \to C$ of the base category do we need to build a conditional of $F$? It turns out, that we need only ask for two simple compatibilities with the $\dagger$-positive component of $F$. We call such a map a \textit{conditional generator} for a $\dagger$-positive map. 

\begin{defi}\label{def:conditional-generator} Given a $\dagger$-positive map $\begin{bmatrix} \alpha & \beta \\
\beta^\dagger & \delta 
\end{bmatrix}: B \oplus C \to B\oplus C$ in $(\mathbb{X}, \dagger)$, a \textbf{conditional generator} for this \dag-positive map is a map $m: B \to C$ in $\mathbb{X}$ such that: 
 \begin{enumerate}[{\em (i)}]
\item $m \circ \alpha = \beta^\dagger$
\item $\delta - m \circ \beta$ is $\dagger$-positive. 
 \end{enumerate}
\end{defi}

\begin{lemma}\label{lemma:m} Let $F= \left( \begin{bmatrix} f \\ g
 \end{bmatrix}, \begin{bmatrix} \alpha & \beta \\
\beta^\dagger & \delta 
\end{bmatrix}, \begin{bmatrix} s \\ t
 \end{bmatrix}  \right): A \to B \otimes C$ be a map in $\mathfrak{G}\left[ (\mathbb{X}, \dagger) \right]_X$, and let $m: B \to C$ be a conditional generator for $\begin{bmatrix} \alpha & \beta \\
\beta^\dagger & \delta 
\end{bmatrix}$. Then define the map $G_m: B \otimes A \to C$ in $\mathfrak{G}\left[ (\mathbb{X}, \dagger) \right]_X$ as the triple: 
 \begin{align}\label{def:cond-m}
 G_m = \left( \begin{bmatrix} m & g -  m \circ f
 \end{bmatrix}, \delta - m \circ \beta, t - m\circ s   \right)
 \end{align}
 Then $G_m$ is a conditional of $F$. 
\end{lemma}
\begin{proof} Since $m$ is a conditional generator, $\delta - m \circ \beta$ is $\dagger$-positive, and so $G_m$ is a well-defined map in $\mathfrak{G}\left[ (\mathbb{X}, \dagger) \right]_X$. Now we need to show that (\ref{eq:conditional}) holds. First since $\alpha$ is $\dagger$-positive and $m \circ \alpha = \beta^\dagger$, we also have that $\alpha \circ m^\dagger = \beta$. Now, by setting $k = g - m\circ f$, $\eta   = \delta - m \circ \beta$, and $u = t - m\circ s$, using  (\ref{longcalc}), we compute that:  
\begin{align*}
&(\mathsf{Id}_B \otimes G_m) \circ (\mathsf{copy}_B \otimes \mathsf{Id}_A) \circ (\mathsf{Id}_B \otimes \mathsf{del}_C \otimes \mathsf{Id}_A) \circ (F \otimes \mathsf{Id}_A) \circ \mathsf{copy}_A \\
&=~  \left( \begin{bmatrix} f \\ m \circ f + k
 \end{bmatrix}, \begin{bmatrix} \alpha & \alpha \circ m^\dagger \\ m \circ \alpha &  \eta + m \circ \alpha\circ m^\dagger \end{bmatrix},  \begin{bmatrix} s \\ u + m \circ s \end{bmatrix} \right) \\
 &=~\fitline{ \left( \begin{bmatrix} f \\ m \circ f + g - m\circ f
 \end{bmatrix}, \begin{bmatrix} \alpha & \alpha \circ m^\dagger \\ m \circ \alpha &  \delta -  m \circ \beta + m \circ \beta \end{bmatrix},  \begin{bmatrix} s \\ t - m\circ s + m \circ s \end{bmatrix} \right)} \\
 &=~ \left( \begin{bmatrix} f \\ g 
 \end{bmatrix}, \begin{bmatrix} \alpha & \alpha \circ m^\dagger \\ m \circ \alpha &  \delta  \end{bmatrix},  \begin{bmatrix} s \\ t  \end{bmatrix} \right) = \left( \begin{bmatrix} f \\ g 
 \end{bmatrix}, \begin{bmatrix} \alpha & \beta \\ \beta^\dagger &  \delta  \end{bmatrix},  \begin{bmatrix} s \\ t  \end{bmatrix} \right) = F
\end{align*}
Therefore $G_m$ satisfies (\ref{eq:conditional}) , thus $G_m$ is a conditional of $F$. 
\end{proof}

Combining the above two lemmas implies that conditionals in the Gauss construction correspond to conditional generators in the base category. 

\begin{theorem}\label{thm:cond=condgen} For a map $F= \left( \begin{bmatrix} f \\ g
 \end{bmatrix}, \begin{bmatrix} \alpha & \beta \\
\beta^\dagger & \delta 
\end{bmatrix}, \begin{bmatrix} s \\ t
 \end{bmatrix}  \right): A \to B \otimes C$ in $\mathfrak{G}\left[ (\mathbb{X}, \dagger) \right]_X$, conditionals of $F$ are in bijective correspondence with conditional generators of $\begin{bmatrix} \alpha & \beta \\
\beta^\dagger & \delta 
\end{bmatrix}$. Explicitly: 
 \begin{enumerate}[{\em (i)}]
\item If $G = \left( \begin{bmatrix} m & k
 \end{bmatrix}, \eta, u  \right)$ is a conditional of $F$ then $m$ is a conditional generator of $\begin{bmatrix} \alpha & \beta \\
\beta^\dagger & \delta 
\end{bmatrix}$;
\item If $m$ is a conditional generator of $\begin{bmatrix} \alpha & \beta \\
\beta^\dagger & \delta 
\end{bmatrix}$, then $G_m$ as defined in (\ref{def:cond-m}) is a conditional of $F$;
 \end{enumerate}
 and these constructions are inverses of each other. 
\end{theorem}
\begin{proof}  Lemma~\ref{lemma:m} tells us precisely how to construct conditionals from conditional generators. Now we prove the converse direction. If $F$ has a conditional $G = \left( \begin{bmatrix} m & k
 \end{bmatrix}, \eta, u  \right)$, then by Lemma~\ref{lem:con-m1}, this implies that $m \circ \alpha = \beta^\dagger$ and $\eta = \delta - m \circ \beta$. However since $\eta$ is $\dagger$-positive (since $G$ is a map in $\mathfrak{G}\left[ (\mathbb{X}, \dagger) \right]_X$), it follows that $\delta - m \circ \beta$ is $\dagger$-positive. Thus $m$ is indeed a conditional generator. Moreover, Lemma~\ref{lem:con-m1} also tells us that $G= \left( \begin{bmatrix} m & g -  m \circ f
 \end{bmatrix}, \delta - m \circ \beta, t - m\circ s   \right) =G_m$, which implies that the above constructions are inverses of each other. 
\end{proof}

Therefore a map $F$ in the Gauss construction has a conditional if and only if its $\dagger$-positive component has a conditional generator $m$. Moreover, every conditional of $F$ is of the form $G_m$ for some conditional generator. Now it is important to stress that, like conditionals, conditional generators need not exist, and even if one did exist, it need not be unique. 

\section{Moore-Penrose Inverses}\label{sec:MP}

The Markov category $\mathsf{Gauss}$ has conditionals; where a canonical choice of conditional can be constructed using the \textit{Moore-Penrose inverse}. In this section, we generalize this result to the setting of \textit{Moore-Penrose dagger categories}  showing that the Gauss construction applied to a Moore-Penrose dagger additive category is a Markov category with conditionals, which is the main result of this paper.

Let us begin by reviewing Moore-Penrose inverses in a dagger category, which is a special kind of \textit{generalized inverse} for maps. For a more in-depth introduction to Moore-Penrose dagger (additive) categories, we invite the reader to see \cite{EPTCS384.10,EPTCS426.3}. 

\begin{defi}\label{def:MP} \cite[Def 2.3]{EPTCS384.10} In a dagger category $(\mathbb{X}, \dagger)$, a \textbf{Moore-Penrose inverse} of a map $f: A \to B$ is a map of dual type $f^\circ: B \to A$ such that the following equalities hold: 
\begin{center}
\setlength{\tabcolsep}{0pt}
\begin{tabular}{@{}p{0.48\linewidth}@{\hspace{0.04\linewidth}}p{0.48\linewidth}@{}}
\textnormal{\bfseries[MP.1]}\; \(f \circ f^\circ \circ f = f\)
&
\textnormal{\bfseries[MP.2]}\; \(f^\circ \circ f \circ f^\circ = f^\circ\)
\\[0.6em]
\textnormal{\bfseries[MP.3]}\; \(\bigl(f\circ f^\circ\bigr)^\dagger = f\circ f^\circ\)
&
\textnormal{\bfseries[MP.4]}\; \(\bigl(f^\circ \circ f\bigr)^\dagger = f^\circ \circ f\)
\end{tabular}
\end{center}

If $f$ has a Moore-Penrose inverse, we say that $f$ is \textbf{Moore-Penrose invertible} or simply \textbf{Moore-Penrose}. A \textbf{Moore-Penrose dagger (additive) category} is a dagger (additive) category such that every map is Moore-Penrose. 
\end{defi}

In arbitrary dagger category, Moore-Penrose inverses need not necessary exist, but if a Moore-Penrose inverse does exists, then it is \emph{unique} \cite[Lemma 2.4]{EPTCS384.10}, so we may speak of \emph{the} Moore-Penrose inverse of a map. Thus for a dagger category, being Moore-Penrose is a property rather than structure. 

\begin{example} It is well-known that complex matrices and real matrices have Moore-Penrose inverses. Indeed, given a $n \times m$ complex matrix $A$, there exists a unique complex matrix $A^\circ$ of size $m \times n$ such that $AA^\circ A = A$, $A^\circ A A^\circ  =A^\circ$, $(AA^\circ)^\ast = AA^\circ$, and $(A^\circ A)^\ast = A^\circ A$. When $A$ is a real matrix, then its Moore-Penrose inverse $A^\circ$ will also be a real matrix, however this time the last two identities can instead be written as $(AA^\circ)^\mathsf{T} = AA^\circ$ and $(A^\circ A)^\mathsf{T} = A^\circ A$. See \cite{campbell2009generalized} for various ways of computing the Moore-Penrose inverse. Therefore it follows that $(\mathsf{MAT}(\mathbb{C}), \ast)$ and $(\mathsf{MAT}(\mathbb{R}), \mathsf{T})$ are Moore-Penrose dagger additive categories\footnote{It is important to note that the choice of dagger here is important, since $(\mathsf{MAT}(\mathbb{C}), \mathsf{T})$ is not Moore-Penrose \cite[Ex 2.10]{EPTCS384.10}.} \cite[Ex 2.9]{EPTCS384.10}. Similarly, quaternionic matrices also have Moore-Penrose inverses~\cite[Thm~5.2]{Wang1998}, therefore \((\mathsf{MAT}(\mathbb{H}), *)\) is a Moore-Penrose dagger additive category as well. 
\end{example}

We now turn towards proving our main objective that we can use Moore-Penrose inverses to always construct conditionals in the Gauss construction. Thanks to Theorem \ref{thm:cond=condgen}, it suffices to explain how to use Moore-Penrose inverses to construct conditional generators. To do so, we first show that in a Moore-Penrose dagger additive category, one of the axioms of a conditional generator can be omitted. 

\begin{lemma}\label{lem:MP-con-gen} In a Moore-Penrose dagger additive category $(\mathbb{X}, \dagger)$, a map $m: B \to C$ is a conditional generator of a $\dagger$-positive map $\begin{bmatrix} \alpha & \beta \\
\beta^\dagger & \delta 
\end{bmatrix}: B \oplus C \to B \oplus C$ if and only if $m \circ \alpha = \beta^\dagger$. 
\end{lemma}
\begin{proof} The $\Rightarrow$ direction is immediate by definition. For the $\Leftarrow$ direction, suppose that $m \circ \alpha = \beta^\dagger$. Then to show that $m$ is a conditional generator we need to show that $\delta - m \circ \beta$ is $\dagger$-positive. To do so, first recall that as explained in the previous section, since  $\begin{bmatrix} \alpha & \beta \\
\beta^\dagger & \delta 
\end{bmatrix}$ is $\dagger$-positive, this implies there exists a map $\begin{bmatrix} \phi & \psi \end{bmatrix}: B \oplus C \to D$, together with maps $\phi: B \to D$ and $\psi: C \to D$, such that the following equalities hold:
\begin{align}\label{eq:phi-psi}
\alpha = \phi^\dagger \circ \phi && \delta = \psi^\dagger \circ \psi && \beta^\dagger = \psi^\dagger \circ \phi && \beta = \phi^\dagger \circ \psi 
\end{align}
Now as we are in a Moore-Penrose  dagger category, we also note that by \cite[Lemma 2.7.(v)]{EPTCS384.10}, the following equality also holds:
\begin{align}\label{eq:alpha-circ}
\alpha^\circ  = \phi^\circ \circ {\phi^\circ}^\dagger
\end{align}
which will be key to showing that $\delta - m \circ \beta$ is $\dagger$-positive. Now since $\alpha$ is $\dagger$-positive and $m \circ \alpha = \beta^\dagger$, we also get that $\alpha \circ m^\dagger = \beta$. Then using the above identity relating the Moore-Penrose inverses of $\alpha$ and $\phi$, we compute: 
\begin{align*}
m \circ \beta &= m \circ \alpha \circ m^\dagger = m \circ \alpha \circ \alpha^\circ \circ \alpha \circ m^\dagger = \beta^\dagger \circ \alpha^\circ \circ \beta \\
&=  \psi^\dagger \circ \phi \circ  \phi^\circ \circ {\phi^\circ}^\dagger \circ \phi^\dagger \circ \psi = \psi^\dagger \circ \phi \circ  \phi^\circ \circ (\phi \circ \phi^\circ)^\dagger \circ \psi  \\
&= \psi^\dagger \circ \phi \circ  \phi^\circ \circ \phi \circ \phi^\circ \circ \psi = \psi^\dagger \circ \phi \circ \phi^\circ \circ \psi 
\end{align*}
Thus $m \circ \beta = \psi^\dagger \circ \phi \circ \phi^\circ \circ \psi$. Now consider the map $\chi := (\mathsf{id}_B - \phi \circ \phi^\circ) \circ \psi$. Then we compute: 
\begin{align*}
\chi^\dagger \circ \chi &=~ \left((\mathsf{id}_B - \phi \circ \phi^\circ) \circ \psi\right)^\dagger \circ (\mathsf{id}_B - \phi \circ \phi^\circ) \circ \psi \\
&=~ \psi^\dagger \circ (\mathsf{id}_B - \phi \circ \phi^\circ)^\dagger \circ (\mathsf{id}_B - \phi \circ \phi^\circ) \circ \psi \\
&=~ \psi^\dagger \circ \left(\mathsf{id}^\dagger_B - (\phi \circ \phi^\circ)^\dagger \right) \circ (\mathsf{id}_B - \phi \circ \phi^\circ) \circ \psi \\
&=~ \psi^\dagger \circ \left(\mathsf{id}_B - \phi \circ \phi^\circ \right) \circ (\mathsf{id}_B - \phi \circ \phi^\circ) \circ \psi \\
&=~ \psi^\dagger \circ \left( \mathsf{id}_B - \phi \circ \phi^\circ - \phi \circ \phi^\circ + \phi \circ \phi^\circ \circ \phi \circ \phi^\circ \right) \circ \psi \\
&=~ \psi^\dagger \circ \left( \mathsf{id}_B - \phi \circ \phi^\circ - \phi \circ \phi^\circ + \phi \circ \phi^\circ \right) \circ \psi \\
&=~ \psi^\dagger \circ \left( \mathsf{id}_B - \phi \circ \phi^\circ \right) \circ \psi \\
&=~ \psi^\dagger \circ \psi - \psi^\dagger \circ \phi \circ \phi^\circ \circ \psi \\
&=~ \delta - m \circ \beta
\end{align*}
So $\delta - m \circ \beta$ is $\dagger$-positive and thus $m$ is a conditional generator. 
\end{proof}

Thanks to Lemma~\ref{lem:MP-con-gen}, we can easily show how to use the Moore-Penrose inverse to build a conditional generator.

\begin{proposition}\label{prop:alpha-circ-gamma} In a Moore-Penrose dagger additive category $(\mathbb{X}, \dagger)$, if $\begin{bmatrix} \alpha & \beta \\
\beta^\dagger & \delta 
\end{bmatrix}: B \oplus C \to B \oplus C$ is $\dagger$-positive, then it has a conditional generator $m = \beta^\dagger \circ \alpha^\circ$. 
\end{proposition}
\begin{proof} By Lemma~\ref{lem:MP-con-gen}, it suffices to show that $m \circ \alpha = \beta^\dagger$. As in the previous proof, since  $\begin{bmatrix} \alpha & \beta \\
\beta^\dagger & \delta 
\end{bmatrix}$ is $\dagger$-positive, this implies there exists  a map $\begin{bmatrix} \phi & \psi \end{bmatrix}: B \oplus C \to D$, together with maps $\phi: B \to D$ and $\psi: C \to D$, such that the equalities in (\ref{eq:phi-psi}) hold. Then using \textbf{[MP.1]}, \textbf{[MP.3]}, and (\ref{eq:alpha-circ}), we compute:
\begin{align*}
m \circ \alpha &= \beta^\dagger \circ \alpha^\circ \circ \alpha =  \psi^\dagger \circ \phi \circ \phi^\circ \circ {\phi^\circ}^\dagger \circ \phi^\dagger \circ \phi = \psi^\dagger \circ \phi \circ \phi^\circ \circ (\phi \circ \phi^\circ)^\dagger \circ \phi \\
&= \psi^\dagger \circ \phi \circ \phi^\circ \circ \phi \circ \phi^\circ \circ \phi = \psi^\dagger \circ \phi \circ \phi^\circ \circ \phi = \psi^\dagger \circ \phi = \beta^\dagger
\end{align*}
So $m \circ \alpha = \beta^\dagger$ and thus $m = \beta^\dagger \circ \alpha^\circ$ is a conditional generator. 
\end{proof}

We may now state the main result of this paper. 

\begin{theorem}\label{thm:MP-Gauss-Markov} Given a Moore-Penrose dagger additive category $(\mathbb{X}, \dagger)$ and any object $X \in \mathbb{X}$, $\mathfrak{G}\left[ (\mathbb{X}, \dagger) \right]_X$ is a Markov category with conditionals. In particular, given a map $F= \left( \begin{bmatrix} f \\ g
 \end{bmatrix}, \begin{bmatrix} \alpha & \beta \\
\beta^\dagger & \delta 
\end{bmatrix}, \begin{bmatrix} s \\ t
 \end{bmatrix}  \right): A \to B \otimes C$ in $\mathfrak{G}\left[ (\mathbb{X}, \dagger) \right]_X$, the map $F\vert_B: B \otimes A \to C$ defined as the triple: 
\begin{align}\label{eq:FvertB}
F\vert_B = \left( \begin{bmatrix} \beta^\dagger \circ \alpha^\circ & g - \beta^\dagger \circ \alpha^\circ \circ f \end{bmatrix}, \delta - \beta^\dagger \circ \alpha^\circ \circ \beta, t - \beta^\dagger \circ \alpha^\circ \circ s \right)
\end{align}
is a conditional of $F$.
\end{theorem}
\begin{proof} By Prop~\ref{prop:alpha-circ-gamma}, we know that $\beta^\dagger \circ \alpha^\circ$ is a conditional generator for $\begin{bmatrix} \alpha & \beta \\
\beta^\dagger & \delta 
\end{bmatrix}$. Therefore, by Lemma~\ref{lemma:m}, it follows that $F$ has a conditional $G_{\beta^\dagger \circ \alpha^\circ}$ as defined in (\ref{def:cond-m}), which is precisely $F\vert_B$. Thus, we conclude that $\mathfrak{G}\left[ (\mathbb{X}, \dagger) \right]_X$ is a Markov category with conditionals.
\end{proof}

\begin{example} Applying Theorem~\ref{thm:MP-Gauss-Markov} to real matrices recaptures the fact that $\mathfrak{G}\left[ (\mathsf{MAT}(\mathbb{R}), \mathsf{T}) \right]_1 \cong \mathsf{Gauss}$ has conditionals (Ex~\ref{ex:Gauss}), which was of course our main motivating example. Furthermore, we also get that $\mathfrak{P}\left[ (\mathbb{X}, \dagger) \right] \cong \mathfrak{G}\left[ (\mathbb{X}, \dagger) \right]_{\mathsf{0}}$ and $\mathfrak{G}\left[ (\mathsf{MAT}(\mathbb{R}), \mathsf{T}) \right]_k \cong \mathsf{Gauss}_k$ are also Markov categories with conditionals. 
\end{example}

\begingroup \sloppy
\begin{example} Applying Theorem~\ref{thm:MP-Gauss-Markov} to complex matrices implies that the Markov category of proper complex Gaussian conditional distributions, $\mathsf{Gauss}^\mathbb{C}_p \coloneqq \mathfrak{G}\left[ (\mathsf{MAT}(\mathbb{C}), \ast) \right]_1$, has conditionals. Similarly, the Markov category of circularly symmetric complex Gaussian conditional distributions, \(\mathfrak{P}(\mathsf{MAT}(\mathbb{C}), \ast)\cong \mathfrak{G}\left[ (\mathsf{MAT}(\mathbb{C}), \ast) \right]_0\), also has conditionals. Similarly for any \(n\in \mathbb{N}\), the Markov category \(\mathfrak{G}\left[ (\mathsf{MAT}(\mathbb{C}), \ast) \right]_n\) has conditionals.
\end{example}
\endgroup

\begin{example}  Applying Theorem~\ref{thm:MP-Gauss-Markov} to quaternionic matrices gives us that the Markov category of proper quaternionic Gaussian conditional distrubtions,  $\mathfrak{G}\left[ (\mathsf{MAT}(\mathbb{H}), \ast) \right]_1 \eqqcolon  \mathsf{Gauss}^\mathbb{H}_p$, has conditionals. Similarly, the Markov category of spherically symmetric quaternionic Gaussian conditional distributions, \(\mathfrak{P}(\mathsf{MAT}(\mathbb{H}), \ast)\cong \mathfrak{G}\left[ (\mathsf{MAT}(\mathbb{H}), \ast) \right]_0\), also has conditionals. Similarly for any \(n\in \mathbb{N}\), the Markov category \(\mathfrak{G}\left[ (\mathsf{MAT}(\mathbb{H}), \ast) \right]_n\) has conditionals.
\end{example}

It would be interesting to apply the Gauss construction to other examples of Moore-Penrose dagger additive categories. For example, we can apply the Gauss construction to non matrix examples of Moore-Penrose dagger additive categories such as on the category of polarized abelian varieties \cite{auffarth2025pseudoinversos}, giving a potential new link from algebraic geometry to probability theory. 

\section{Other Constructions of Conditionals}

In the previous section we showed how to use Moore-Penrose inverses to build conditionals in the Gauss construction. However, eagle-eyed readers may have noticed that we did not need the full power of the Moore-Penrose inverse to do so! In this section, we explain how we can use slightly weaker generalized inverses to build conditionals in the Gauss construction. We begin by introducing the analogue of $(i,j,k)$-generalized inverses \cite[Def 6.2.4]{campbell2009generalized} in a dagger category.

\begin{defi}\label{def:MP-ijk} Let $\lbrace i_{1} < \hdots < i_{j} \rbrace \subseteq \lbrace 1 < 2 < 3 < 4 \rbrace$. Then in a dagger category $(\mathbb{X}, \dagger)$, a \textbf{[MP.$i_1,\hdots,i_j$]-inverse} of a map $f: A \to B$ is a map of dual type $f^\bullet: B \to A$ which satisfies \textbf{[MP.$i_1$]}, ..., \textbf{[MP.$i_j$]}. 
\end{defi}

Notice of course that a \textbf{[MP.1,2,3,4]}-inverse is precisely the Moore-Penrose inverse. We will be particularly interested with \textbf{[MP.1,3]}-inverses. Note that while Moore-Penrose inverses are unique, \textbf{[MP.1,3]}-inverses need not be unique. Indeed, note that a Moore-Penrose inverse is always a \textbf{[MP.1,3]}-inverse, however, even in a Moore-Penrose dagger category, a map $f$ can have a \textbf{[MP.1,3]}-inverse $f^\bullet$ which is not its Moore-Penrose inverse $f^\circ$. For example for complex matrices, \textbf{[MP.1,3]}-inverses correspond to \emph{least squares inverses} \cite[Fig 6.1]{campbell2009generalized} which is used to give least squares solutions \cite[Def 2.1.1]{campbell2009generalized}, and a complex matrix can have multiple least squares inverses. 

Our objective is to show that we can build conditionals in the Gauss construction using \textbf{[MP.1,3]}-inverses. For the remainder of this section, we work in a dagger additive category $(\mathbb{X}, \dagger)$. 

\begin{lemma}\label{lemma:phi-bullet} Let $\begin{bmatrix} \alpha & \beta \\
\beta^\dagger & \delta 
\end{bmatrix}: B \oplus C \to B\oplus C$ be a $\dagger$-positive map in $(\mathbb{X}, \dagger)$ such that there is a map $\begin{bmatrix} \phi & \psi \end{bmatrix}: B \oplus C \to D$ such that $\begin{bmatrix} \alpha & \beta \\
\beta^\dagger & \delta 
\end{bmatrix}=\begin{bmatrix} \phi & \psi \end{bmatrix}^\dagger \circ \begin{bmatrix} \phi & \psi \end{bmatrix}$ for which $\phi: B \to D$ has a \textbf{[MP.1,3]}-inverse $\phi^\bullet: D \to B$. Then $m = \psi^\dagger \circ {\phi^\bullet}^\dagger$ is a conditional generator of $\begin{bmatrix} \alpha & \beta \\
\beta^\dagger & \delta 
\end{bmatrix}$. 
\end{lemma}
\begin{proof} The computations in this proof are similar to those in the proof of Lemma~\ref{lem:MP-con-gen} but there are some subtle differences. Now by assumption, recall that $\phi$ and $\psi$ satisfy (\ref{eq:phi-psi}). Also by assumption, $\phi^\bullet$ being a \textbf{[MP.1,3]}-inverse of $\phi$ implies that is satisfies the following:
\begin{align*}
    \phi \circ \phi^\bullet \circ \phi = \phi &\quad & (\phi \circ \phi^\bullet)^\dagger = \phi \circ \phi^\bullet
\end{align*}
We need to show that $m = \psi^\dagger \circ {\phi^\bullet}^\dagger$ satisfies the two necessary properties to be a conditional generator. We first compute that: 
\begin{align*}
m \circ \alpha = \psi^\dagger \circ {\phi^\bullet}^\dagger \circ \phi^\dagger \circ \phi = \psi^\dagger \circ (\phi \circ \phi^\bullet)^\dagger \circ \phi = \psi^\dagger \circ \phi \circ \phi^\bullet \circ \phi = \psi^\dagger \circ \phi = \beta^\dagger 
\end{align*}
So $m \circ \alpha = \beta^\dagger$ as desired. Next we need to show that $\delta - m \circ \beta$ is $\dagger$-positive. As before, since $m \circ \alpha = \beta^\dagger$, we also get that $\beta = \alpha \circ m^\dagger$. So we compute:  
\begin{align*}
m \circ \beta &= m \circ \alpha \circ m^\dagger = m \circ \phi^\dagger \circ \phi \circ m^\dagger  = \psi^\dagger \circ {\phi^\bullet}^\dagger \circ \phi^\dagger \circ \phi \circ \phi^\bullet \circ \psi \\
&=  \psi^\dagger \circ (\phi \circ \phi^\bullet)^\dagger \circ \phi \circ \phi^\bullet \circ \psi 
\psi^\dagger \circ \phi \circ \phi^\bullet \circ \phi \circ \phi^\bullet \circ \psi =
\psi^\dagger \circ \phi \circ \phi^\bullet \circ \psi
\end{align*}
Thus $m \circ \beta = \psi^\dagger \circ \phi \circ \phi^\bullet \circ \psi$. Then consider the map $\zeta = (\mathsf{id}_B - \phi \circ \phi^\bullet) \circ \psi$. So we compute that:
\begin{align*}
\zeta^\dagger \circ \zeta &=~ \left((\mathsf{id}_B - \phi \circ \phi^\bullet) \circ \psi\right)^\dagger \circ (\mathsf{id}_B - \phi \circ \phi^\bullet) \circ \psi 
\\&=~ \psi^\dagger \circ (\mathsf{id}_B - \phi \circ \phi^\bullet)^\dagger \circ (\mathsf{id}_B - \phi \circ \phi^\bullet) \circ \psi \\
&=~ \psi^\dagger \circ \left(\mathsf{id}^\dagger_B - (\phi \circ \phi^\bullet)^\dagger \right) \circ (\mathsf{id}_B - \phi \circ \phi^\bullet) \circ \psi \\
&=~ \psi^\dagger \circ \left(\mathsf{id}_B - \phi \circ \phi^\bullet \right) \circ (\mathsf{id}_B - \phi \circ \phi^\bullet) \circ \psi \\
&=~ \psi^\dagger \circ \left( \mathsf{id}_B - \phi \circ \phi^\bullet - \phi \circ \phi^\bullet + \phi \circ \phi^\bullet \circ \phi \circ \phi^\bullet \right) \circ \psi \\
&=~ \psi^\dagger \circ \left( \mathsf{id}_B - \phi \circ \phi^\bullet - \phi \circ \phi^\bullet + \phi \circ \phi^\bullet \right) \circ \psi \\
&=~ \psi^\dagger \circ \left( \mathsf{id}_B - \phi \circ \phi^\bullet \right) \circ \psi \\
&=~ \psi^\dagger \circ \psi - \psi^\dagger \circ \phi \circ \phi^\bullet \circ \psi \\
&=~ \delta - m \circ \beta 
\end{align*}
So $\delta - m \circ \beta$ is indeed $\dagger$-positive. Therefore, we conclude that $m = \psi^\dagger \circ {\phi^\bullet}^\dagger$ is a conditional generator. 
\end{proof}

Therefore, it follows that if every map in our underlying dagger category has a \textbf{[MP.1,3]}-inverse, the Gauss construction has conditionals. 

\begin{proposition} If every map in $(\mathbb{X}, \dagger)$ admits a \textbf{[MP.1,3]}-inverse, then for every object $X \in \mathbb{X}$, $\mathfrak{G}\left[ (\mathbb{X}, \dagger) \right]_X$ is a Markov category with conditionals.
\end{proposition}

Of course every map in a Moore-Penrose dagger additive category admits a \textbf{[MP.1,3]}-inverse, namely the Moore-Penrose inverse. However, it is important to again stress that a map could have a \textbf{[MP.1,3]}-inverse which is not the Moore-Penrose inverse. Thus Lemma~\ref{lemma:phi-bullet} will in general produce a different conditional generator to that of Prop~\ref{prop:alpha-circ-gamma}. This is to be expected, since conditionals are far from being unique. However, if we use the Moore-Penrose inverse to take the role of the \textbf{[MP.1,3]}-inverse in Lemma~\ref{lemma:phi-bullet}, we do get back our favourite conditional generator of Prop~\ref{prop:alpha-circ-gamma}. 

\begin{lemma} Suppose that $(\mathbb{X}, \dagger)$ is Moore-Penrose. Let $\begin{bmatrix} \alpha & \beta \\
\beta^\dagger & \delta 
\end{bmatrix}$ be a $\dagger$-positive map in $(\mathbb{X}, \dagger)$. Then for any map $\begin{bmatrix} \phi & \psi \end{bmatrix}$ such that $\begin{bmatrix} \alpha & \beta \\
\beta^\dagger & \delta 
\end{bmatrix}=\begin{bmatrix} \phi & \psi \end{bmatrix}^\dagger \circ \begin{bmatrix} \phi & \psi \end{bmatrix}$, we have that $\beta^\dagger \circ \alpha^\circ = \psi^\dagger \circ {\phi^\circ}^\dagger$. 
\end{lemma}
\begin{proof} Using \textbf{[MP.2]}, (\ref{eq:phi-psi}), and (\ref{eq:alpha-circ}), we easily compute that: 
\begin{align*}
    \beta^\dagger \circ \alpha^\circ &= \psi^\dagger \circ \phi \circ \phi^\circ \circ {\phi^\circ}^\dagger =  \psi^\dagger \circ (\phi \circ \phi^\circ)^\dagger \circ {\phi^\circ}^\dagger \\
    &= \psi^\dagger \circ (\phi^\circ \circ \phi \circ \phi^\circ)^\dagger =  \psi^\dagger \circ {\phi^\circ}^\dagger
\end{align*}
Thus $\beta^\dagger \circ \alpha^\circ = \psi^\dagger \circ {\phi^\circ}^\dagger$ as desired. 
\end{proof} 

We have made a big deal to emphasize that conditional generators and conditionals are not unique in general; we conclude this paper with a class of suitable positive maps which have a unique conditional generator. 

\begin{lemma} Take a Moore-Penrose dagger category $(\mathbb{X}, \dagger)$ with  a \dag-positive map $\begin{bmatrix} \alpha & \beta \\
\beta^\dagger & \delta 
\end{bmatrix}$ such that $\alpha$ is an isomorphism. Then $\beta^\dagger \circ \alpha^{-1}$ is its unique conditional generator. Therefore, for every object $X \in \mathbb{X}$, if $F= \left( \begin{bmatrix} f \\ g
 \end{bmatrix}, \begin{bmatrix} \alpha & \beta \\
\beta^\dagger & \delta 
\end{bmatrix}, \begin{bmatrix} s \\ t
 \end{bmatrix}  \right)$ is a map in $\mathfrak{G}\left[ (\mathbb{X}, \dagger) \right]_X$, then $F\vert_B$ as defined in (\ref{eq:FvertB}) is its unique conditional. 
\end{lemma}
\begin{proof} Recall that in a Moore-Penrose dagger category, the Moore-Penrose inverse of an isomorphism is its inverse \cite[Lemma 2.7.(ii)]{EPTCS384.10}. Thus since $\alpha$ is an isomorphism, we have that $\alpha^\circ = \alpha^{-1}$. So by Lemma~\ref{lem:MP-con-gen}, we know that $\beta^\dagger \circ \alpha^{-1}$ is a conditional generator. Now if we had another conditional generator $m$, then by definition we have that $m \circ \alpha = \beta^\dagger$. However since $\alpha$ is an isomorphism, we get that $m = \beta^\dagger \circ \alpha^{-1}$, and so we conclude that $\beta^\dagger \circ \alpha^{-1}$ is the unique conditional generator. The rest of the statement then follows as well. 
\end{proof}

\bibliographystyle{plain}      
\bibliography{Gaussbib}   

\end{document}